\documentclass[11pt]{article}
\usepackage[utf8]{inputenc}
\usepackage{paralist}
\usepackage[margin=3.0cm]{geometry} % Changes the margins of the paper.
\usepackage{graphicx} % Simple graphics.

\usepackage{amssymb, amsfonts, amsmath, amscd, amsthm} % Symbols, fonts, diagrams, and environments from AMS.
\usepackage{mathtools} % Symbols such as inclusion arrows etc.
\usepackage{titlesec} % Section and subsection fonts.
\usepackage{csquotes} % Quotation marks.
\usepackage[shortlabels]{enumitem} % Changes line spacing within lists.
\usepackage{tikz-cd} % Diagrams.
\usepackage{wrapfig} % To get text to flow around images.
\usepackage{epigraph} % Quotations at the beginning of chapters.
\usepackage{hyperref} % Hyperlinks.
\usepackage{todonotes} % To make margin notes
\usepackage{comment}
\usepackage{float}
\usepackage{mathrsfs} % For \mathscr
\usepackage{soul}
\usepackage[margin=1cm, labelfont=bf, font=small]{caption}
\usepackage{datetime}
\usepackage{booktabs}

\newdateformat{monthyeardate}{\monthname[\THEMONTH] \THEYEAR}

\DeclareFontFamily{U}{mathx}{\hyphenchar\font45}
\DeclareFontShape{U}{mathx}{m}{n}{
      <5> <6> <7> <8> <9> <10>
      <10.95> <12> <14.4> <17.28> <20.74> <24.88>
      mathx10
      }{}
\DeclareSymbolFont{mathx}{U}{mathx}{m}{n}
\DeclareFontSubstitution{U}{mathx}{m}{n}
\DeclareMathAccent{\widecheck}{0}{mathx}{"71}
\DeclareMathAccent{\wideparen}{0}{mathx}{"75}

\hypersetup{
     colorlinks   = true,
     citecolor    = green
}

\numberwithin{equation}{section}

\titleformat*{\section}{\Large \scshape\center}
\titleformat*{\subsection}{\fontsize{14}{14} \sffamily}

\theoremstyle{plain}
\newtheorem{theorem}{Theorem}[section]
\newtheorem*{theorem*}{Theorem}

\newtheorem{lemma}[theorem]{Lemma}
\newtheorem{proposition}[theorem]{Proposition}
\newtheorem{corollary}[theorem]{Corollary}

\theoremstyle{definition}
\newtheorem{definition}[theorem]{Definition}
\newtheorem*{definition*}{Definition}
\newtheorem{example}[theorem]{Example}

\theoremstyle{remark}
\newtheorem*{remark}{Remark}

\usepackage{tikz}
\usetikzlibrary{cd}

\newcommand{\Ad}{\mathrm{Ad}}

\newcommand{\ad}{\mathrm{ad}}

\newcommand{\FKO}{\mathcal{F}_{\mathrm{KO}}}
\newcommand{\FW}{\mathcal{F}_{\mathrm{W}}}
\newcommand{\dmr}{\,\mathrm{d}\mu_r}
\newcommand{\dml}{\,\mathrm{d}\mu_l}
\newcommand{\dm}{\,\mathrm{d}\mu}

\newcommand{\tr}{\operatorname{tr}}
\newcommand{\D}{\mathcal{D}}

\newcommand{\SF}{\mathrm{Stab}(F)}
\newcommand{\sF}{\mathrm{stab}(F)}

\title{Weyl Quantization of Exponential Lie Groups for Square Integrable Representations}
\author{Stine Marie Berge and Simon Halvdansson}
\date{\monthyeardate\today}

\newcommand{\R}{\mathbb{R}}

\newcommand{\Addresses}{{% additional braces for segregating \footnotesize
  \bigskip
  \footnotesize

  \textsc{Department of Mathematical Sciences, Norwegian University of Science and Technology, Norway.}\par\nopagebreak
  \textit{E-mail addresses}: 
 \texttt{stine.m.berge@ntnu.no} and \texttt{simon.halvdansson@ntnu.no}
}}

\begin{document}

\maketitle

\begin{abstract}
    We construct a general quantization procedure for square integrable functions on well-behaved connected exponential Lie groups. The Lie groups in question should admit at least one co-adjoint orbit of maximal possible dimension. The construction is based on composing the Fourier-Wigner transform with another Fourier transform we call the Fourier-Kirillov transform. This quantization has many desirable properties including respecting function translations and inducing a well-behaved Wigner distribution. 
    
    Moreover, we investigate the connection to the operator convolutions of quantum harmonic analysis. This is intricately connected to Weyl quantization in the Weyl-Heisenberg setting. We find that convolution relations in quantum harmonic analysis can be written as group convolutions of Weyl quantizations. This implies that the squared modulus of the wavelet transform of the representation can be written as a convolution between two Wigner distributions. Lastly, we look at how we can extend known results based on Weyl quantization to wider classes of groups using our quantization procedure.
\end{abstract}

\section{Introduction}
The Weyl quantization mapping, originally studied by H. Weyl in \cite{Weyl1927, Weyl1950}, is an object of fundamental importance in analysis and mathematical physics. One view of the mapping is as a correspondence rule between functions on phase space, $L^2(\R^{2d})$, and observables in the form of Hilbert-Schmidt operators on the Hilbert space $L^2(\R^d)$. There is a large body of influential work studying this mapping, including but not limited to \cite{Simon1992, Pool1966, deGosson2011}. We refer to \cite{Wong1998} for a comprehensive overview.

Various generalizations of the Weyl quantization to different phase spaces and function spaces have been proposed and studied over the last decades. Some generalizations focus on the dequantization of rank-one operators, that is, dequantization of pure states. This perspective has been applied to a wide variety of contexts, see e.g.\ \cite{Chen2012, Peng2008, ghosh2023unbounded, ali2000, RACHDI2003}. Following the work in \cite{gayral2007fourier,Berge2022affine} we set out to formulate Weyl quantization in the general context in a more systematic way. This formulation is based on the composition of the Fourier-Wigner transform and a \emph{Fourier-Kirillov transform}, with the Fourier-Kirillov transform taking the role of the symplectic Fourier transform. In this way, we are able to formulate our results in a general setting without reference to the specific group structure. It should be mentioned that the proposed Weyl quantization given is similar in  spirit to the one given in \cite{Mantoiu2019}, and enjoys many of the same properties. 

\subsubsection*{Main construction}
The goal of this article is to set up a quantization scheme for functions on a exponential Lie group $G$. The proposed scheme generalizes the Weyl quantization for the affine group, see \cite{gayral2007fourier, gazeau20192d, gazeau2016covariant, Berge2022affine, Berge2022}. Integral to the quantization is the existence of square integrable representations. From Kirillov orbit theory, it is well known how to construct the irreducible unitary representations $\pi: G \to \mathcal{U}(\mathcal{H})$ from the co-adjoint orbits. When the dimension of the co-adjoint orbit coincides with the dimension of the group, then the corresponding representation is square integrable. Hence in our construction we will always assume that the orbits are of the same dimension as the group. Notably, this includes the affine group \cite{Berge2022} and the shearlet group \cite{Kutyniok2012}. 

As mentioned above, our construction of Weyl quantization is based on generalizing the relation $a_S = \mathcal{F}_\sigma(\FW(S))$ where $\mathcal{F}_\sigma$ is the symplectic Fourier transform and $\FW$ is the Fourier-Wigner transform. The Fourier-Wigner transform is a standard object in this setup and it and its inverse are defined as
\begin{align*}
    \FW(A)(x) = \tr(A \D \pi(x)),\qquad \FW^{-1}(f) = \pi(\check{f}) \circ \D.
\end{align*}
where $\check{f}(x) = f(x^{-1})$, $\D$ is the inverse Duflo-Moore operator and $\pi(f)$ is the (left) integrated representation of $f : G \to \mathbb{C}$. We will properly define and give background on these objects in the preliminaries section. Section \ref{sec:FW} is devoted to working out all of the properties of the Fourier-Wigner transform which we will need, including that it can be extended to a unitary isometry from the Hilbert-Schmidt operators $\mathcal{S}^2(\mathcal{H})$ to $L^2_r(G)$, the space of square integrable functions on $G$ with respect to the right Haar measure.

In place of the symplectic Fourier transform, we will set up a \emph{Fourier-Kirillov transform} $\FKO$. This transform is coherently linked to a chosen co-adjoint orbit. Let $K(x)$ denote the co-adjoint map and $\mathcal{O}_F$ the co-adjoint orbit $\{K(x)F:x\in G\}$. Then $\FKO:L^2_r(G)\to L^2_r(G)$ is defined as
\begin{align*}
    \FKO(f)(x)&=\frac{1}{\sqrt{|\mathrm{Pf}_F|\cdot\Delta(x)}}\int_{G}f(y)e^{2\pi i\langle K(x^{-1})F,\, \log(y)\rangle}\frac{1}{\sqrt{\Theta(\log(y))}}\dmr(y)
\end{align*}
where $\mathrm{Pf}_F$, $\Delta$ and $\Theta$ are functions which ensure proper normalization across spaces to be further detailed later. The Fourier-Kirillov transform can be viewed, up to some normalizing factors, as the Fourier transform on $L^2(\mathfrak{g})\to L^2(\mathfrak{g}^*)$ restricted to the orbit corresponding to the element $F\in \mathcal{O}_F$.

With the two Fourier transforms defined, we define the quantization $A$ and dequantization $a = A^{-1}$  by composing them as
\begin{align*}
    A_f = \FW^{-1}(\FKO^{-1}(f)),\qquad a_S = \FKO(\FW(S)).
\end{align*}
Using the properties of the transforms, we are able to deduce that the full quantization map $A : L^2_r(G) \to \mathcal{S}^2(\mathcal{H})$ is a unitary isometry.
This quantization scheme is similar to the Wigner map outlined in \cite{TWAREQUEALIS2003Piau,TWAREQUEALIS2000TWff}, however the inclusion of the co-adjoint map creates an algebraic structure on the phase space, namely the original group structure.
\subsubsection*{Quantization structure}
Having proved that the quantization is a unitary isometry, we move on to showing that quantization respects translations and complex conjugation in an appropriate manner. Specifically, we show the two identities
\begin{align}\label{eq:intro_translate_conjugate}
    \pi(x)^*A_f \pi(x) = A_{R_{x^{-1}}f},\qquad A_f^* = A_{\bar{f}},
\end{align}
where $R_{x^{-1}}f(y) = f(yx^{-1})$. These properties hold true for classical Weyl quantization but are not given for quantization schemes based on the Wigner distribution. Moreover, classical Weyl quantization can be given the additional structure of being an isometric $*$-isomorphism between $H^*$ algebras as was shown by Pool \cite{Pool1966}. This means that taking the adjoint and composing operators in $\mathcal{S}^2$ corresponds to certain actions on functions in a isomorphic way. We show that the same is true for our quantization which lead to notions of \emph{twisted convolutions} and \emph{twisted multiplication}. Specifically, we show that all maps are isometric $*$-isomorphisms between $H^*$ algebras in the commutative diagram
\begin{equation}\nonumber
    \begin{tikzcd}
        \big(\mathcal{S}^2(\mathcal{H}),\, \circ,\,{}^*\big) \arrow[d, "\FW"] \arrow[rd, dotted, "a"]&\\
        \big(\FW(\mathcal{S}^2),\, \natural,\,\sqrt{\Delta(\cdot)}\, \overline{ \check{\phantom{a}} }\big) \arrow[r,"\FKO"]& \big(L^2_r(G),\, \sharp,\, \overline{ \phantom{a} }\big)
    \end{tikzcd}
\end{equation}
where the intermediate space $\FW(\mathcal{S}^2) = \FKO^{-1}(L^2_r(G))$ is a subset of $L^2_r(G)$. This construction is detailed in Section \ref{sec:algebraic_structure}.

\subsubsection*{Wigner distributions}
Wigner distributions are most commonly related to Weyl quantization through the weak relation
\begin{align*}
    \big\langle f, W(\phi, \psi) \big\rangle_{L^2_r} = \big\langle A_f \psi, \phi \big\rangle_{\mathcal{H}}
\end{align*}
for all $f \in L^2_r(G)$ and $\psi, \phi \in \mathcal{H}$. This relation is equivalent to the Wigner distribution $W(\psi, \phi)$ being the dequantization of the rank-one operator $(\psi \otimes \phi) : \xi \mapsto \langle \xi, \phi \rangle \psi$. For this reason, we set $W(\psi, \phi) = a_{\psi \otimes \phi}$ and get the properties that the mapping $(\psi, \phi) \mapsto W(\psi, \phi)$ is sesquilinear, that $\overline{W(\psi, \phi)} = W(\phi, \psi)$ and $R_x W(\psi, \phi) = W(\pi(x)\psi, \pi(x)\phi)$ from general quantization properties such as \eqref{eq:intro_translate_conjugate} for free.
In the case that the group is unimodular, we show that
\begin{align*}
    \int_G W(\psi, \phi)(x)\,\mathrm{d}\mu(x) = \langle \psi, \phi \rangle.
\end{align*}

\subsubsection*{Quantum harmonic analysis}
The framework of quantum harmonic analysis is concerned with convolutions, translations and Fourier transforms of operators and their interactions. One of several ways to define these operations is through their interaction with Weyl quantization. In recent years, quantum harmonic analysis has been developed for the affine group \cite{Berge2022} in conjunction with a Weyl quantization procedure \cite{Berge2022affine} as well as for general locally compact groups \cite{Halvdansson2023} without any connection to quantization. We will show that our quantization procedure is compatible with the quantum harmonic analysis operations in \cite{Berge2022} and \cite{Halvdansson2023}, thus imbuing them with additional structure. Specifically, we will show that the operator convolutions from quantum harmonic analysis can be realized as group convolutions of Weyl symbols as
\begin{align*}
    f \star S &= A_{f * a_S},\\
    T \star S &= a_T * \check{a_S}.
\end{align*}
As a consequence, we can generalize relations such as the convolution of two Wigner distributions being a spectrogram,
\begin{align}\label{eq:wigner_wigner_spectrogram}
    W(\psi) * \widecheck{W(\phi)} = |\mathcal{W}_\phi \psi|^2
\end{align}
where $\mathcal{W}_\phi \psi$ is the wavelet transform of $\psi$ with respect to $\phi$.

\subsubsection*{Applications}
Since our construction is a generalization of that for the Weyl-Heisenberg and affine groups, for applications our hope is to generalize some of the properties of those Weyl quantization. The first application in Section \ref{sec:applications} has not been explored for the affine group but uses the relation \eqref{eq:wigner_wigner_spectrogram} to develop a criterion for \emph{phase retrieval}, meaning inversion of the map $\mathcal{W}_\phi \psi \mapsto |\mathcal{W}_\phi \psi|$.

Having developed a Wigner distribution, we treat the \emph{Wigner approximation problem} which asks how close a given $L^2_r(G)$ function is to being a Wigner distribution using the same approach as has been used for the Weyl-Heisenberg and affine groups earlier. Similarly, we are able to extend a proof technique which shows that for both Wavelet spaces $\mathcal{W}_\phi(\mathcal{H})$ and Wigner spaces $W(\mathcal{H}, \phi)$, spaces induced by different $\phi_1, \phi_2$ have trivial intersection as long as $\phi_1$ and $\phi_2$ are linearly independent. In both of these cases, we have show how general Weyl quantization allows us to apply the proof techniques from standard Weyl quantization to more general contexts.

These examples of applications are not meant to be exhaustive but rather to illustrate the value of the tools developed in the article.

\section{Preliminaries}
\subsection{Weyl quantization on the Weyl-Heisenberg Group}\label{sec:prelim_weyl}
To set the stage, we will review how Weyl quantization is set up in the well-known case of the Weyl-Heisenberg group. Specifically, we will outline how we view Weyl quantization as being induced by a pair of Fourier transforms instead of the Wigner distribution. This viewpoint is more amenable to generalizations. For a more physics-oriented perspective, see e.g. \cite[Chap.~13]{hall2013quantum}.

%For the purposes of this section, the Weyl-Heisenberg group refers to the reduced Weyl-Heisenberg group which we will write as $\R^{2d} \ni (x, \omega)$ where $x, \omega \in \R^d$. 

\subsubsection{Via Wigner distribution}
The original construction by Weyl is most easily formulated using the \emph{cross-Wigner distribution} \cite{wigner1932quantum}, also called the \emph{Wigner transform}, given by
\begin{align*}
    W(\psi, \phi)(x, \omega) = \int_{\R^d} \psi(x+t/2)\overline{\phi(x-t/2)}e^{-2\pi i \omega \cdot t}\,\mathrm{d}t, \qquad \psi, \phi \in L^{2}(\mathbb{R}^{d}).
\end{align*}
We write $W(\phi)=W(\phi,\phi)$ for convenience. Using the Wigner transform, we can define the Weyl quantization $A_f$ of a function $f \in L^2(\R^{2d})$ weakly via the relation
\begin{align}\label{eq:wigner_weak_def}
    \langle A_f \psi, \phi \rangle_{L^2(\R^d)} = \big\langle f, W(\phi, \psi) \big\rangle_{L^2(\R^{2d})}.
\end{align}
We say that the map $ f \mapsto A_f$ is the \emph{quantization map} or \emph{Weyl transform} and write $S \mapsto a_S$ for the inverse, called the \emph{dequantization map}. The dequantization map is a bijective isometry from the space of Hilbert-Schmidt operators $\mathcal{S}^2$ to $L^2(\R^{2d})$, see \cite{Pool1966}. Hence $A_{a_S} = S$ and $f = a_{A_f}$. Through an elementary computation, one can show that the following \emph{Moyal identity} holds for the Wigner transform.
\begin{theorem}\label{theorem:moyal_wigner}
    For all $\psi_1, \psi_2, \phi_1, \phi_2 \in L^2(\R^d)$, it holds that
    \begin{align}\label{eq:moyal_wigner}
        \big\langle W(\psi_1, \phi_1), W(\psi_2, \phi_2) \big\rangle_{L^2(\R^{2d})} = \langle \psi_1, \psi_2\rangle_{L^2(\R^{d})} \overline{\langle \phi_1, \phi_2\rangle}_{L^2(\R^{d})}.
    \end{align}
\end{theorem}
In view of this result, the quantization rule \eqref{eq:wigner_weak_def} implies that the dequantization of the rank-one operator $\psi \otimes \phi$ is the Wigner transform $W(\psi, \phi)$. %Indeed, we can verify this as
%\begin{align*}
 %   \langle (\psi \otimes \phi) \xi, \eta \rangle_{L^2(\R^{d})} = \langle \xi, \phi \rangle_{L^2(\R^{d})} \langle \psi, \eta \rangle_{L^2(\R^{d})} = \big\langle W(\psi, \phi), W(\eta, \xi) \big\rangle_{L^2(\R^{2d})}.
%\end{align*}
Consequently, one can view quantization as being induced by the Wigner transform or vice versa. 

Define the (projective) \emph{Schr\"odinger representation} by
\[
    \pi(z) f(t) = \pi(x,\omega) f(t) = M_\omega T_x f(t) = e^{2\pi i \omega t} f(t-x)
\]
where we have made the identification $z = (x, \omega) \in \R^{2d}$. We will discuss the time and frequency-shifts $T_x, M_\Omega$ further in Section \ref{sec:prelim_TF}. 
The key relation between the Schr\"odinger representation and the Weyl quantization is included in the following proposition:
\begin{proposition}\label{prop:classic_quantization_props}
    Let $f \in L^2(\R^{2d})$. Then $A_{\bar{f}} = A_f^*$ and 
    \begin{equation*}
        A_{f(\cdot - z)} = \pi(z)A_f \pi(z)^*.
    \end{equation*}
\end{proposition}

\subsubsection{Via Fourier transforms}\label{sec:prelim_fw_fko}
When generalizing Weyl quantization beyond the standard $\R^{2d}$ setting, a common practice is to generalize the definition of the Wigner distribution and use \eqref{eq:wigner_weak_def} to induce a quantization map. In this paper, we argue that while the relation \eqref{eq:wigner_weak_def} indeed is fundamental to quantization, directly defining a different version of the Wigner distribution is not. In fact, there is another way to define Weyl quantization which itself induces a Wigner distribution which is what we will detail in this section. 

Recall that phase space $\R^{2d}$ is a symplectic space when equipped with the symplectic form 
\[
\sigma(x_1, \omega_1, x_2, \omega_2) =\omega_1 \cdot x_2 - \omega_2 \cdot x_1, \qquad (x_1, \omega_1),\, (x_2, \omega_2)\in \mathbb{R}^{2d}.
\]
Here the appropriate Fourier transform is the \emph{symplectic Fourier transform} $\mathcal{F}_\sigma$, defined for $f \in L^{1}(\mathbb{R}^{2d})$ as
\begin{align*}
    \mathcal{F}_\sigma(f)(z) = \int_{\R^{2d}} f(z') e^{-2\pi i \sigma(z,z')}\,\mathrm{d}z'.
\end{align*}
As is the case with the standard Fourier transform, the symplectic Fourier transform can be extended to square integrable functions as well. On the space of operators, we will instead use the Fourier-Wigner (sometimes called Fourier-Weyl) transform. The \emph{Fourier-Wigner transform} $\FW$ is a map from the trace-class operators on $L^{2}(\mathbb{R}^{d})$, denoted by $\mathcal{S}^1$, to a subset of $L^2(\R^{2d})$, given by
\begin{align*}
    \FW(S)(z) = e^{-\pi i x \cdot \omega} \tr(\pi(-z)S).
\end{align*}
%where \[\pi(z) f(t) = \pi(x,\omega) f(t) = M_\omega T_x f(t) = e^{2\pi i \omega t} f(t-x)\] is a time-frequency shift which we will discuss further in Section \ref{sec:prelim_TF}. It too can be extended to a map from all of $\mathcal{S}^2$.
It follows from \cite{folland1989harmonic, Wong1998, Luef2018} that the Weyl quantization can be written as
\begin{align*}
    a_S = \mathcal{F}_\sigma(\FW(S)).
\end{align*}

\subsection{Representation theory on locally compact groups}\label{sec:prelim_rep_theory}
In this section, let $G$ denote a locally compact group of \textit{type I}, see e.g.~\cite[p.~229]{folland2016course}. In the context of Lie groups, ``type I" is often referred to as \emph{tame groups}, in contrast with \emph{wild groups}. Tame groups are a broad class of Lie groups that include all exponential Lie groups, see \cite{take57,GOTOMorikuni1978Oaco}. Given such a group we will denote the \textit{left Haar measure} on the group $G$ by $\mu_l^G$ and the \textit{right Haar measure} by $\mu_r^G$, see \cite{folland1989harmonic} for the definition. There exists a multiplicative function $\Delta_G:G\to \mathbb{R}^+$ called the \textit{modular function} defined by the relationship
\begin{align*}
    \mu_l^G(x)=\Delta_G(x) \mu_r^G(x), \qquad x \in G.    
\end{align*}
Whenever it is clear from the context, we will suppress the group $G$ in the notation $\mu_l^G$, $\mu_r^G$ and $\Delta_G$. The right versus left Haar measure on an $L^p$-space will be indicated by either an $r$ or $l$ subscript.

Given functions $f_1,\, f_2\in L^1_r(G)$ we define their \textit{(right-)convolution} by
\begin{align*}
    f_1*f_2(x)=\int_{G}f_1(y)f_2(xy^{-1})\dmr(y).  
\end{align*}
Define \textit{right}- and \textit{left-translation} of a function $f:G\to \mathbb{C}$ by $(R_y f)(x)=f(xy)$ and $(L_yf)(x)=f(y^{-1}x)$, respectively.
The \textit{involution} of a function $f\in L^1_l(G)$ is defined by 
\begin{align*}
    \check{f}(x)=f(x^{-1}),\quad \text{ for } x\in G.
\end{align*}
Notice that the involution maps $L^1_l(G)$ to $ L^1_r(G)$. The inverse of the involution will also be denoted by $\check{\cdot}$. It should also be noted that different definitions of function involution exists in the literature, see e.g.\ \cite{folland2016course}.

We will denote by $(\pi, \mathcal{H}_\pi)$ a \textit{irreducible unitary representation} $\pi$ of the group $G$ acting on the separable Hilbert space $\mathcal{H}_\pi$. When it is clear from the context, we will write $\mathcal{H}$ instead of $\mathcal{H}_\pi$.  The representation $(\pi, \mathcal{H})$ induces an \textit{integrated (left) representation} acting on $f\in L^1_l(G)$ by
\begin{align*}
    \pi(f)\phi=\int_G f(x)\pi(x)\phi \dml(x),\quad \text{ for }\phi\in \mathcal{H}. 
\end{align*}
Given an irreducible unitary representation $(\pi,\mathcal{H})$, we denote the \textit{wavelet transform} by 
\begin{align}\label{eq:wavelet_def}
    \mathcal{W}_{\phi}\psi(x)=\langle \psi, \pi(x)^*\phi\rangle_\mathcal{H}.
\end{align}
The convention of using $\pi(x)^*$ instead of $\pi(x)$ for the wavelet transform is nonstandard but more compatible with the right Haar measure which we will prefer. We say that $(\pi,\mathcal{H})$ is \textit{square integrable} if there exists a non-zero $\phi \in \mathcal{H}$ such that $\mathcal{W}_{\phi}\phi\in L^2_r(G)$. For square integrable representations, we have the following classical orthogonality relation from \cite{duflo1976}.
\begin{theorem}[Duflo-Moore Theorem]\label{theorem:duflo-moore}
    Let $(\pi , \mathcal{H})$ be a unitary square integrable representation. There exists a unique positive densely defined operator $\D^{-1}:\mathrm{Dom}(\D^{-1})\subset \mathcal{H}\to \mathcal{H}$ such that
    \begin{equation}\label{eq:Duflo-Moore ortogonality}
        \langle \mathcal{W}_{\phi_1}\psi_1, \mathcal{W}_{\phi_2}\psi_2\rangle_{L^2_r}=\langle \psi_1, \psi_2\rangle_\mathcal{H} \overline{\langle \D^{-1} \phi_1,\D^{-1}\phi_2 \rangle_\mathcal{H}},
    \end{equation}
    for $\phi_1, \phi_2\in \mathrm{Dom}(\D^{-1})$ and $\psi_1, \psi_2\in \mathcal{H}$.
\end{theorem}
We will refer to the operator $\D^{-1}$ in the theorem as the \textit{Duflo-Moore operator}. Functions $\phi \in \operatorname{Dom}(\D^{-1})$ are said to be \emph{admissible}. Below we collect some standard results on how the modular function $\Delta$, representation $\pi$, and Duflo-Moore operator $\D^{-1}$ interact.
\begin{lemma}\label{lemma:haar_measure_properties}
    The following properties hold:
    \begin{enumerate}[label=(\roman*)]
        \item $d\mu_l(x y) = \Delta(y)\dml(x)$,
        \item $d\mu_r(x) = \Delta(x^{-1}) \dml(x)$,
        \item $d\mu_r(x^{-1}) = \Delta(x) \dmr(x)$,%=\mathrm{d}\mu_l(x)$,
%        \item $\Delta(xy) = \Delta(x) \Delta(y)$,
%        \item $\Delta(x^{-1}) = \frac{1}{\Delta(x)}$,
        \item $\D\pi(x) = \sqrt{\Delta(x)} \pi(x)\D$.
    \end{enumerate}
\end{lemma}
\subsection{Lie groups}\label{sec:prelim_lie}
In this section, we will assume that $G$ is a connected Lie group with Lie algebra $\mathfrak{g}$. We will denote the (Lie group) \textit{exponential map} by $\exp:\mathfrak{g}\to G$. In the case that the exponential map is a diffeomorphism, we say that the Lie group $G$ is \textit{exponential}. The inverse of the exponential map is called the logarithm and is denoted by $\log$.
For a Lie group homeomorphism $\Phi:G\to H$ the following diagram commutes 
\begin{equation}\label{exponential map and homeo}
     \begin{tikzcd}
\mathfrak{g} \arrow{r}{\Phi_*} \arrow[swap]{d}{\exp} & \mathfrak{h} \arrow{d}{\exp} \\%
G \arrow{r}{\Phi}& H
\end{tikzcd},
\end{equation}
where $\mathfrak{g}$ and $\mathfrak{h}$ are the Lie algebras of $G$ and $H$, respectively.

Let $\mathrm{d}X$ denote the Euclidean measure on $\mathfrak{g}$. There exists a function $\Theta:\mathfrak{g}\to \mathbb{R}^+$ such that 
\[\mathrm{d}\mu_r(\exp(X))=\Theta(X)\,\mathrm{d}X\]
is the right Haar measure in exponential coordinates. Additionally, 
\[\mathrm{d} \mu_l(\exp(X))=\Theta(-X)\,\mathrm{d}X.\]
Hence 
\begin{equation}\label{eq:modular quotion}
\frac{\Theta(-X)}{\Theta(X)}=\Delta_G(\exp(X)).
\end{equation}
We can compute $\Theta$ by using the formula 
\begin{equation}\label{eq:volume exponential}
\Theta(X)=\left|\det\left(\frac{e^{\mathrm{ad}(X)}-1}{\mathrm{ad}(X)}\right)\right|.
\end{equation}
In the case that $G$ is a nilpotent group, we have that $\Theta(X)=1$ for all $X \in \mathfrak{g}$.

%We define the \textit{commutator series} inductively by $\mathfrak{g}^0:=\mathfrak{g}$, and $\mathfrak{g}^{l+1}=[\mathfrak{g}^{l},\mathfrak{g}^{l}]$. We say that a Lie group is \textit{solvable} if the commutator series eventually vanish, i.e., there exists a $k$ such that $\mathfrak{g}^k=0$. An important special case of solvable groups are the \textit{nilpotent groups} and exponential groups. In all that follows we will assume that our group $G$ is exponential, unless otherwise specified. Exponential groups are strictly in between nilpotent and solvable (connected and simply connected) Lie groups.

A Lie group can act on itself via conjugation $\Psi_{\underline{}}:G\to \mathrm{Aut}(G)$ given by $\Psi_x(y)=xyx^{-1}$. Fixing $x\in G$, the \emph{adjoint representation} is given by 
\[\Ad_x=(\Psi_x)_*:\mathfrak{g}\to \mathfrak{g}.\]
The modular function can be written by using the adjoint map as
\begin{equation}\label{eq:modular_function_adjoint}
    \det(\Ad_x)=\Delta(x).
\end{equation}
By using \eqref{exponential map and homeo} we get that 
$\exp(\Ad_x X)=\Psi_x(\exp(X))$. Fixing $Y\in \mathfrak{g}$ and taking the derivative in the other variable, we get that $\ad_Y:\mathfrak{g}\to \mathfrak{g}$ can be computed as $\ad_Y X=[Y, X]$. %Using \eqref{exponential map and homeo} we get \[\Ad_{\exp(Y)}X=\exp(\ad_Y)X=X+[Y, X]+\frac{1}{2!}[Y,[Y, X]]+\ldots.\]

\subsubsection{Group examples}\label{sec:group_examples_prelim}
Before going further, it is instructive to look at a couple of examples of exponential Lie groups where the representations are known.
\begin{example}[Heisenberg Group]\label{ex:heisenberg}
The Heisenberg group is the underlying group for the standard Wigner transform outlined in Section \ref{sec:prelim_TF}. It should however be noted that the representation given by co-adjoint orbit theory is not square integrable.

Let $\mathbb{H}$ denote the Heisenberg group, which when written in matrix form is the group 
\[\mathbb{H}=\left\{(x,y,z)=\begin{pmatrix}
1&x&z\\
0&1&y\\
0&0&1\\
\end{pmatrix}: x,y,z\in \mathbb{R}\right\}.
\]
%The inverse matrix is given by 
%\[\begin{pmatrix}
%1&x&z\\
%0&1&y\\
%0&0&1\\
%\end{pmatrix}^{-1}=\begin{pmatrix}
%1&-x&-z+yx\\
%0&1&-y\\
%0&0&1\\
%\end{pmatrix}.
%\]
The corresponding Lie algebra $\mathfrak{h}$ is the span of 
\[X=\begin{pmatrix}
0&1&0\\
0&0&0\\
0&0&0\\
\end{pmatrix},\quad
Y=\begin{pmatrix}
0&0&0\\
0&0&1\\
0&0&0\\
\end{pmatrix},\quad
Z=\begin{pmatrix}
0&0&1\\
0&0&0\\
0&0&0\\
\end{pmatrix}.
\]
It is endowed with the bracket relations
\[[X,Z]=[Y,Z]=0 \text{ and } [X,Y]=Z.\]
The exponential map is given by 
\[\exp(xX+yY+zZ)=(x,y,z+xy/2).\]
It is well-known that both the right and the left Haar-measure is given by 
\[\mathrm{d}\mu^\mathbb{H}(x,y,z)=\mathrm{d} x\,\mathrm{d} y\,\mathrm{d} z.\]
Hence the modular function $\Delta_\mathbb{H}(x,y,z)= 1$ and $\Theta(xX+yY+zZ)=1$.
Except the characters, the irreducible representations on $\mathbb{H}$ is, up to equivalence, given by
 \[\pi_\hbar (x,y,z)f(w)=\exp(2\pi i \hbar( z+yw))f(w+x),\]
where $\hbar\in \mathbb{R}\setminus \{0\}$ and $f\in L^2(\mathbb{R})$.
\end{example}
\begin{example}[Affine Group]\label{ex:affine_group}
The \textit{(reduced) affine group} $(\mathrm{Aff}, \cdot_{\mathrm{Aff}})$ is the Lie group whose underlying set is the upper half plane $\mathrm{Aff} \coloneqq \mathbb{R}^{+} \times \mathbb{R} \coloneqq  (0, \infty)\times\mathbb{R} $, while the group operation is given by \[(a,x) \cdot_{\mathrm{Aff}} (b,y) \coloneqq (ab,ay + x), \qquad (a,x), (b,y) \in \mathrm{Aff}.\]
We can represent the affine group $\mathrm{Aff}$ and its Lie algebra $\mathfrak{aff}$ in matrix form 
\[\mathrm{Aff} = \left\{\begin{pmatrix} a & x \\ 0 & 1\end{pmatrix} : a> 0, \, x \in \mathbb{R}\right\}, \quad \mathfrak{aff} = \left\{uU+vV : u,v \in \mathbb{R}\right\},\]
where \[U=\begin{pmatrix}
        1&0\\0&0 \end{pmatrix} \text{ and }V=\begin{pmatrix}
        0&1\\0&0 \end{pmatrix}.\]
The Lie algebra structure of $\mathfrak{aff}$ is completely determined by 
\begin{equation*}
    \left[U, V\right] = V.
\end{equation*}

The right and left Haar measure are given by \[\dmr(a,x)=\frac{\mathrm{d}a\,\mathrm{d}x}{a},\qquad \dml(a,x)=\frac{\mathrm{d}a\,\mathrm{d}x}{a^2},\]
hence the modular function is $\Delta(a,x)=\frac{1}{a}$.
Computing the exponential map gives \[\exp(uU+vV)=(e^u,v\lambda(u)),\]
where $\lambda(u)=\frac{e^u-1}{u}$.
In the case of the exponential map $\Theta$ is given by \[\Theta(uU+vV)=e^{-u}\cdot\lambda(u).\]
Except the characters, there are only two non-equivalent irreducible representations of the affine group. These are given by
\begin{equation*}
    \pi_+(a,x)\psi(r) \coloneqq e^{-2\pi i x r}\psi(ar), \qquad \psi \in L^{2}(\mathbb{R}^{+}, r^{-1}\, dr)
\end{equation*}
and
\begin{equation*}
    \pi_-(a,x)\psi(r) \coloneqq e^{2\pi i x r}\psi(ar), \qquad \psi \in L^{2}(\mathbb{R}^{+}, r^{-1}\, dr).
\end{equation*}
\end{example}

\begin{example}[Shearlet Group]\label{ex:shearlet}
A matrix representation of the Shearlet group $\mathbb{S}$ is given by 
    \begin{align*}
       (a,s,x_1,x_2)=(a,s,x)= \begin{pmatrix}
            a & \sqrt{a}s & x_1\\
            0 & \sqrt{a} & x_2\\
            0&0&1
	  \end{pmatrix},\quad a>0,\,s\in \mathbb{R}, x\in \mathbb{R}^2.
    \end{align*}
    The multiplication of two elements is defined as
    \begin{align*}
        (a,s,x_1,x_2)(b,t,y_1,y_2) &= (ab, s+\sqrt{a}t, x+S_sA_ay)\\
        &=(ab, s+\sqrt{a}t, x_1+ay_1+\sqrt{a}sy_2,x_2+\sqrt{a}y_2)
    \end{align*}
    where
    \begin{align*}
      A_a=\begin{pmatrix}
	a&0\\
	0&\sqrt{a}
      \end{pmatrix},\qquad 
      S_s=\begin{pmatrix}
	1&s\\
	0&1
      \end{pmatrix}.
    \end{align*}
%   Computing the inverse element of $(a,s,x_1,x_2)$ gives
%    \[\left(\frac{1}{a}, -\frac{s}{\sqrt{a}}, \frac{sx_2-x_1}{a}, -\frac{x_2}{\sqrt{a}}\right).\]
    The left- and right-Haar measure are given by 
    \[
    \mathrm{d}\mu_l(a,s,x_1,x_2)=\frac{\mathrm{d}a\,\mathrm{d}s\,\mathrm{d}x_1\,\mathrm{d}x_2}{a^3},\qquad \mathrm{d}\mu_r(a,s,x_1,x_2)=\frac{\mathrm{d}a\,\mathrm{d}s\,\mathrm{d}x_1\,\mathrm{d}x_2}{a}.
    \]
    Hence the modular function is 
    \[\Delta_{\mathbb{S}}(a,s,x)=\frac{1}{a^2}.\]

The Lie algebra in matrix form is given by
\begin{equation*}
    \mathfrak{s}=\left\{  \begin{pmatrix}
        \alpha & \sigma & \xi_1\\
        0&\alpha/2&\xi_2\\
        0&0&0
    \end{pmatrix}\, :\, \alpha,\sigma,\xi_1,\xi_2\in \mathbb{R} \right\}.
\end{equation*}
Let $A,\, B,\, C$ and $D$ be a basis for the Lie algebra consisting of the matrices
        \[A=\begin{pmatrix}
        1&0&0\\
        0&1/2&0\\
        0&0&0
        \end{pmatrix},\,B=\begin{pmatrix}
        0&1&0\\
        0&0&0\\
        0&0&0
        \end{pmatrix},\,C=\begin{pmatrix}
        0&0&1\\
        0&0&0\\
        0&0&0
        \end{pmatrix},\text{ and }D=\begin{pmatrix}
        0&0&0\\
        0&0&1\\
        0&0&0
        \end{pmatrix}.\]
The Lie bracket is given by
\begin{multline*}
[\alpha A +\sigma B+ \xi_1C+\xi_2D,\tilde{\alpha}A+ \tilde{\sigma}B +\tilde{\xi_1}C+\tilde{\xi_2}D]\\=\frac{\alpha\tilde{\sigma}-\tilde{\alpha}\sigma}{2}B+ (\alpha \tilde{\xi_1}-\tilde{\alpha }\xi_1 +\sigma\tilde{\xi_2}-\tilde{\sigma}\xi_2)C+ \frac{\alpha\tilde{\xi_2}-\tilde{\alpha}\xi_2}{2}D.
\end{multline*}
In this case we have the exponential map
\[\exp(\alpha A+\sigma B+\xi_1C+\xi_2D)=(e^\alpha,\sigma \lambda(\alpha/2),\xi_1\lambda(\alpha)+\sigma \xi_2\lambda(\alpha/2)^2/2,\xi_2\lambda(\alpha/2)),\]
where \[\lambda(\alpha)=\frac{e^\alpha -1}{\alpha}.\]
%while the logarithm is given by 
%\[\log(a,s,x_1,x_2)=\left(\log(a),\frac{s}{\lambda(\log(a)/2)}, \frac{x_1-sx_2/2}{\lambda(\log(a))},\frac{x_2}{\lambda(\log(a)/2)}\right).\]
%The adjoint map is given by 
%\[\Ad_{(a,s,x_1,x_2)}(\alpha A+\sigma B+\xi_1C+\xi_2D)=\alpha A+( \sqrt{a}\sigma-\frac{s\alpha}{2})B+( \frac{\alpha s x_2}{2}-x_1\alpha-\sqrt{a}x_2\sigma+\xi_1a+\sqrt{a}s\xi_2)C+(\sqrt{a}\xi_2 -\frac{x_2\alpha}{2})D. \]
The function $\Theta$ giving the measures in exponential coordinates is given by \[\Theta(\alpha A+\sigma B+\xi_1C+\xi_2D)=\lambda(\alpha)\lambda(\alpha/2)^2.\]
The irreducible representations of central importance are
\[\pi_+(a,s,x_1,x_2)\phi(b,t)=e^{-2\pi i(bx_1+\sqrt{b}tx_2)}\phi\left(ba, t+s\sqrt{b}\right)\]
and 
\[\pi_-(a,s,x_1,x_2)\phi(b,t)=e^{2\pi i(bx_1+\sqrt{b}tx_2)}\phi\left(ba, t+s\sqrt{b}\right)\]
where $\phi \in L^2(\mathbb{R}^{+}\times \mathbb{R}, \frac{\mathrm{d}b\,\mathrm{d}t}{b})$.

\end{example}
\begin{comment}
    
\begin{example}[Engel Group]
The Engel group $\mathbb{E}$ is the nilpotent group consists of matrices on the form 
\[(a,b,c,d)_{\mathbb{E}}=\begin{pmatrix}
1&b&c&d\\
0&1&a&a^2/2\\
0&0&1&a\\
0&0&0&1
\end{pmatrix}
, \quad a,b,c,d\in \mathbb{R}.\]
It is well known that any nilpotent group is unimodular.
The centre of the group is given by 
\[Z=\left\{\begin{pmatrix}
1&0&0&d\\
0&1&0&0\\
0&0&1&0\\
0&0&0&1
\end{pmatrix}
\colon \quad d\in \mathbb{R}\right\}.\]
The Lie algebra consists of matrices on the form 
\[(\alpha,\beta,\gamma,\delta)_{\mathfrak{e}}=\begin{pmatrix}
0&\beta&\gamma&\delta\\
0&0&\alpha&0\\
0&0&0&\alpha\\
0&0&0&0
\end{pmatrix}
, \quad \alpha,\beta,\gamma,\delta\in \mathbb{R}.\]
\end{example}

The Lie bracket is given by 
\[[(\alpha_1,\beta_1,\gamma_1,\delta_1)_{\mathfrak{e}},(\alpha_2,\beta_2,\gamma_2,\delta_2)_{\mathfrak{e}}]=(0,0,\alpha_2\beta_1-\alpha_1\beta_2,\alpha_2\gamma_1-\alpha_1\gamma_2)_{\mathfrak{e}}.\]
The exponential map is given by
\[\exp((\alpha,\beta,\gamma,\delta)_{\mathfrak{e}})=(\alpha, \beta, \gamma+\frac{\alpha \beta}{2},\delta+\frac{\alpha \gamma}{2}+\frac{\alpha^2\beta}{6})_{\mathbb{E}}.\]
The adjoint map is given by \[\Ad\left((a,b,c,d)_{\mathbb{E}}\right)(\alpha,\beta,\gamma,\delta)_{\mathfrak{e}}=(\alpha,\beta,- a \beta + \alpha b + \gamma,\frac{a^{2} \beta}{2} - a \left(\alpha b + \gamma\right) + \alpha c + \delta)_{\mathfrak{e}}.\]
\end{comment}

\subsection{Representation theory of exponential Lie groups}\label{sec:prelim_rep_theory_exp_lie}

\subsubsection{Induced representations on Lie groups}
In this section, let $G$ be a Lie group and $H$ a closed subgroup of $G$. Then a \textit{right-coset} is the set $Hg=\{hg:h\in H\}$ for all $g\in G$. The collection of all right cosets will be denoted by $R=H\setminus G$. There exists a quotient map $q:G\to R$, sending $g\mapsto Hg$. It is well known that $R$ is a homogeneous manifold equipped with a canonical measure inherited from the Lie group. The measure on $R$ has the property that for every continuous compactly supported function $f$ on $G$ we have
\[\int_G f(g) \dmr^G(g)=\int_{R}\int_H f(hg)\frac{\Delta_G(h)}{\Delta_H(h)}\dmr^H(h) \,\mathrm{d}\mu^{R}(g). \]
Using this measure we can define the space $L^2(R)$ of all square integrable functions. 
%Whenever we are given a quotient group $H\setminus G$ we define the convolution between two functions by 
%\[f*g(x)=\int_{H\setminus G}f(y)g(xy^{-1})\,d\mu_{H/G}(y).\]
%Notice that $f*g$ is again a function on $H\setminus G$. For proof and more detailed exposition \cite[Sec.~2.1]{fuhr2004} or \cite[App.~III \& V]{kirillov2004lectures}.
Whenever we are given a function $f\in L^1_r(G)$ we have that 
\[\tilde{f}(Hx)=\int_Hf(hx)\frac{\Delta_G(h)}{\Delta_H(h)} \dmr^H(h)\]
is in $L^1(R)$.

Given a section $s:U\subset R\to G$ define the function $h_s:R\times G\to H$ by 
\begin{equation*}
 s(x)g=h_s(x,g)s(xg).
\end{equation*}
The function $h_s$ has the additional property that 
\[h_s(x,g_1g_2)=h_s(x,g_1)h_s(xg_1,g_2).\]
Let $(\pi,\mathcal H)$ be a representation of the group $H$. 
Then we define the induced representation acting on $L^2(R,\mathcal{H})$ by 
\[(\mathrm{ind}_H^G\pi(g)f)(x)= \sqrt{\frac{\Delta_H(h_s(x,g))}{\Delta_G(h_s(x,g))}}\pi(h_s(x,g)) f(x\cdot g).\]

An important special case is when $\pi$ is a character, i.e.,\ a representation acting on the Hilbert space $\mathbb{C}$. Let $F\in \mathfrak{h}^*$. Then when $H$ is an exponential group we have that \[\pi(h)=e^{-2\pi i\langle F,\log(h)\rangle}\]
is a character. Then the induced representation acts on
$L^2(R)$ by 
\[(\mathrm{ind}_H^G\pi(g)f)(x)= \sqrt{\frac{\Delta_H(h_s(x,g))}{\Delta_G(h_s(x,g))}}e^{-2\pi i\langle F,\log_H(h_s(x,g))\rangle} f(x\cdot g).\]

\subsubsection{Semi-direct product}
In the case that $G$ can be written as the semi-direct product of $R$ and $H$, the formulas for induced representation simplify. The definition of semi-direct product relies on the existence of a group homomorphism $\phi: R\to \mathrm{Aut}(H)$ such that the product on $G=R\times H$ is given by 
\[(r,h)\cdot (s,j)=(rs,h\phi(r)j).\]
We will denote the group endowed with the product given by the semi-direct product with respect to $\phi$ by $R\rtimes_\phi H$.
%The inverse of an element is given by \[(a,x)^{-1}=(a^{-1}, \phi(a^{-1})x^{-1}).\]
Then $R$ can be identified with the right cosets $H\setminus G$. Notice that 
\[(e_R, \phi(a) y)=(a,e_H)(e_R, y)(a,e_H)^{-1}=\Psi_{(a,e_H)}(e_R, y).\]
%More generally, we see that $H$ when identified with $(e_R,H)$ is a normal subgroup of $G$, since 
%\[\Psi_{(a,x)}(e_R,y)=\Psi_{(e_R,x)}(e_R, \phi(a) y)=\Psi_{(e_R,x)}(e_R, x(\phi(a) y)x^{-1}).\]

Since $H$ is a normal subgroup, we have that $\Delta_H=\Delta_G$.
The semi-direct product have the structure of a quotient group with the right Haar-measure given by
\[\mathrm{d}\mu_r^G = \mathrm{d}\mu^R_r\,\mathrm{d}\mu_r^H.\]
Set $s:R\to G$ to be $a\mapsto (a,e_H)$. Then $h_s(a,(b,y))=(e_R,\phi(a)y)$.
We have that $\Delta_H=\Delta_G$, hence the induced representation simplifies to 
\[(\mathrm{ind}_H^G\pi((b,y))f)(a)= \pi(\phi(a)y) f(a\cdot b).\]
If $(\pi,\mathbb{C})$ is the character
\[\pi(h)=e^{-2\pi iF(\log(h))}\]
defined on $H$, then the induced representation is given by 
\begin{equation}\label{semi:representation}
(\mathrm{ind}_H^G\pi((b,y))f)(a)= e^{-2\pi i F\log_G(\phi(a)y)} f(a\cdot b)=e^{-2\pi i F\log_H(\phi(a)y)} f(a\cdot b)
\end{equation}
acting on $L^2(R,d\mu^R_r)$.
The modular function is given by 
\[\Delta_G(a,x)=\frac{1}{|\mathrm{det}_H(\phi(a)_*)|}\Delta_R(a)\Delta_H(x)=\frac{1}{|\mathrm{det}_H(\Ad_a)|}\Delta_R(a)\Delta_H(x) .\]
In the case that $H$ is unimodular, we have that $\Delta_G(a,x)=\Delta_G(a,e_H)$. In this case, we set 
\[(\D^{-1}\phi)(r)=\sqrt{\Delta_G(r,e_H)}\phi(r)=\frac{\sqrt{\Delta_R(r)}}{\sqrt{|\mathrm{det}_H(\Ad_r)|}}\phi(r).\] %Hence we have that 
%\[\pi(g)\D^{-1}=\sqrt{\Delta_G(g)} \D^{-1}\pi(g).\]

%In the case that $G=H\rtimes_\phi R$ is the semi-direct product between $R$ and $H$ the induced representation has a simpler form. Let $s:H\to G$ be the injection $s(h)=(e_R,h)$ and $\phi:R\to \mathrm{Aut}(H)$ defined by $\phi_r(h)=(e_R,r)(h,e_H)(e_R,r)^{-1}$. In this case we have that the function $h_s:R\times G\to H$ by 
%\begin{equation}\label{eq:h_s}
%h_s(r,(r',h'))= \phi_r(h').
%\end{equation}
%Hence we have that the representation is given by 
%\[(\mathrm{ind}_H^G\pi(r,h)f)(r')= \exp(-2\pi i\langle F, \log(\phi_{r'}(h)\rangle) f(r'\cdot r).\]

\subsubsection{Kirillov co-adjoint orbit theory}\label{sec:co-adjoint}
The objective of Kirillov orbit theory is to show how the co-adjoint orbits and the representations are linked.
The Kirillov orbit method gives us an explicit way to construct all irreducible representation on certain Lie groups. As the name implies, the method states that for every co-adjoint orbit there exists a unique, up to equivalence, irreducible representation associated to it. Moreover, the theory states that all irreducible representations can be constructed in such a way. We will now give a short introduction of how the irreducible representations are constructed from the orbits. For a complete overview, see \cite{ArnalDidier2020RoSL,kirillov2004lectures}.

Associated to the adjoint map we can define the \textit{co-adjoint map} $K(g):\mathfrak{g}\to \mathfrak{g}$ by the equation
\[\langle K(g)F, X\rangle=\langle F, \Ad_{g^{-1}}X\rangle.\]
We will denote the stabilizer of this map by $\mathrm{Stab}(F)=\{g\in G:K(g)F=F\}$.
The derivative of the co-adjoint map is given by the equation 
\[\langle K_*(X)F, Y\rangle=\langle F, -\ad_{X}Y\rangle.\]
Given a $F$ we denote the \textit{co-adjoint orbit} by $\mathcal{O}_F=K(G)F\subset\mathfrak{g}$. We can identify the orbit with the quotient $\mathrm{Stab}(F)\setminus G =\{\SF g:g\in G\}$ by using the map $\kappa:\SF\setminus G\to \mathcal{O}_F$ defined as $\kappa(x)=K(x^{-1})F$, where $x\in \SF\setminus G$. 
In the case that $\SF=\{e\}$, the orbit can be identified with the original group $G$.
In this case, we will denote the inverse of $\kappa$ by $\kappa^{-1}:\mathcal{O}_F\to G$.

 \subsubsection{Symplectic structure of the orbit}\label{sec:orbit}
We will take a \emph{symplectic space} to mean a (smooth) manifold $M$ equipped with a closed non-degenerate $2$-form $\omega:\wedge^2 (T^*M)\to \mathbb{R}$. The form $\omega$ is referred to as the \emph{symplectic form}. All symplectic manifolds are even dimensional, hence we will denote the dimension of $M$ by $2d$. The simplest example of a symplectic manifold is $\mathbb{R}^{2d}$ endowed with the standard symplectic form \[\omega_{(x_1,...,x_d,y_1,...,y_d)}=\sum_{i=1}^d\mathrm{d}x_i\wedge \mathrm{d}y_i,\]
where $(x_1,...,x_d,y_1,...,y_d)\in \mathbb{R}^{2d}$ are the standard coordinates. The \emph{volume form} on a symplectic manifold is given by \[\mathrm{d}\omega = \frac{\omega^d}{d!}.\]

A diffeomorphism $f:M_1\to M_2$ between two symplectic manifolds $(M_j,\omega_j)$ is called a \emph{symplectomorphism} if $f^* (\omega_2)=\omega_1$, where $f^*$ denotes the pullback of the function $f$. If such a map $f$ exists between two symplectic manifolds, we say that the manifolds are \emph{symplectomorphic}.
Darboux's theorem states that every symplectic manifold of dimension $2d$ is locally symplectomorphic to $\mathbb{R}^{2d}$ endowed with the standard symplectic form.

%Let $S\subset TM$, where $(M,\omega)$ is a symplectic manifold. Denote by $S_x\subset T_xM$, then 
%\[S_p^\perp = \{x\in T_pM:\omega_p(x,y)=0\,\forall\, y\in S_p\}.\]
%Then\todo{double "then" here, maybe merge sentences with a "and"?} $S^\perp=\cup_{p\in M} S_p^\perp$.

Let $N\subset M$ be a submanifold of the symplectic manifold $(M,\omega)$. Then we say that $N$ is \emph{isotropic} if $\left.\omega\right|_{TN}=0$.
A \emph{Lagrangian manifold} is an isotropic manifold of maximal dimension $\mathrm{dim}(N)=\frac{1}{2}\mathrm{dim}(M)$. Alternatively, we can define a \emph{Lagrangian manifold} to be an isotropic manifold where $\omega_{TN^\perp}=0$. We define a \emph{Lagrangian fibration} to be a fibration where all the fibers are Lagrangian manifolds. Locally, any Lagrangian filtration can be written as $(x_1,\ldots,x_n,y_1,\ldots, y_n)$ where the symplectic form can be written as \[\omega_{(x_1,...,x_n,y_1,...,y_n)}=\sum_{j=1}^n\mathrm{d}x_j\wedge \mathrm{d}y_j.\]

%An almost complex structure is a vector bundle isomorphism $J:TM\to TM$ satisfying $J^2=-1$. We say that the symplectic form is $J$ invariant if \[\omega(JX,JY)=\omega(X,Y).\] 

 One can induce a symplectic form on the co-adjoint orbit $\mathcal{O}_F\sim \mathrm{Stab}(F)\setminus G$ as follows:
Define the symplectic form at the point $F\in \mathcal{O}_F$ by \[\omega_F(K_*(F)X,K_*(F)Y)=\langle F, [X,Y]\rangle ,\]
where $X,Y\in \mathfrak{g}/\sF$ that can be identified with $T_F^*\mathcal{O}_F$. We can define the symplectic form on the rest of the $T^*\mathcal{O}_F$ by using the group action on the orbit. Then the symplectic form is defined by $\omega_{\kappa(x)}=\kappa(x)^*\omega_F$. Hence by definition, the symplectic form becomes right invariant with respect to the group action.

A \emph{invariant complex-polarization} with respect to $\omega_F$ is a subalgebra $\mathfrak{h}$ of $\mathfrak{g}\otimes\mathbb{C}$ such that 
\begin{itemize}
    \item $\mathfrak{h}$ is a Lagrangian subspace of $\omega_F$.
    \item $\mathfrak{h}+\overline{\mathfrak{h}}$ is a subalgebra of $\mathfrak{g}\otimes\mathbb{C}$.
    \item $\Ad(s)\mathfrak{h}=\mathfrak{h}$ for every $s\in \mathrm{Stab}(F)$. \end{itemize}
Recall that we say that a subalgebra $\mathfrak{k}$ of $\mathfrak{g}\otimes \mathbb{C}$ is real if \[\mathfrak{k}=\overline{\mathfrak{k}}.\]
We let $\mathfrak{e}$ denote the largest real subalgebra of $\mathfrak{g}$, i.e.\ 
\[\mathfrak{e}=(\mathfrak{h}+\overline{\mathfrak{h}})\cap \mathfrak{g}.\] 
Additionally, we define the subalgebra $\mathfrak{d}$ to be the  \[\mathfrak{d}=(\mathfrak{h}\cap\overline{\mathfrak{h}})\cap \mathfrak{g}.\]
Notice that in the case that $\mathfrak{h}$ is a real subalgebra, we have that 
\[\mathfrak{e}=\mathfrak{d}.\]
\begin{definition}
    We say that an invariant complex polarization satisfies the Pukanszky condition if
    \[K(\exp(\mathfrak{d}))F=F+\mathrm{An}(\mathfrak{e}),\]
    where $\mathrm{An}(\mathfrak{e})$ denotes the annihilator of $\mathfrak{e}$.
\end{definition}

\subsubsection{Representations from co-adjoint orbits}\label{sec:rep_co_adj_orb}

In this section we will let $G$ be a exponential Lie group. Let $\mathcal{O}$ be a co-adjoint orbit containing a point $F\in \mathcal{O}$. Find a subgroup of maximal dimension $H$ with Lie algebra $\mathfrak{h}$ such that $\langle F, [\mathfrak{h},\mathfrak{h}]\rangle=0$. It can be shown that $\dim(H\setminus G)=\frac{1}{2}\dim(\mathcal{O})$. Define the character on $H$ by $\chi_F(h) = e^{2\pi i\langle F, \log(h)\rangle}$. Then the induced representation \[\pi_F=\mathrm{ind}_H^G\chi_F\]
is an irreducible representation on $G$.
In general, these representations are only square integrable when $\SF$ consists of only the identity element. %Hence we will assume that $\SF=\{e\}$ in the rest of the paper.

\begin{example}[Heisenberg Group]
Continuing with the notation from Example \ref{ex:heisenberg}, let $X^*,\, Y^*$ and $Z^*$ be the dual basis of $X$, $Y$ and $Z$. 
Then we have that the co-adjoint map satisfies \[K(x,y,z)(aX^*+ bY^*+ cZ^*)=(a+yc)X^*+ (b-xc)Y*+cZ^*.\]
From the co-adjoint map, we see that we have two kinds of orbits. Setting $F=\hbar Z^*$ for $\hbar\not = 0$ gives the two-dimensional orbits
\[\mathcal{O}_\hbar=\{aX^*+ bY^*+\hbar Z^*:a,b\in \mathbb{R}\}.\]
When $F=aX^*+ bY^*$ we get the zero-dimensional orbits
\[\mathcal{O}_{(a,b)}=\{aX^*+ bY^* : a, b \in \R\}.\]

The representations corresponding to the orbits $\mathcal{O}_\hbar$ is $\pi_{-\hbar}$. Since the $\mathrm{Stab}(\hbar Z^*)=(0,0,z)$, this representation is not square integrable.
\end{example}
\begin{example}[Affine Group]
    The affine group has two two-dimensional orbits. Continuing with the notation from Example \ref{ex:affine_group}, let $U^*$ and $V^*$ be the dual basis to the Lie algebra consisting of $U$ and $V$.
        The co-adjoint map is given by \[K(a,x)(uU^*+vV^*)=(u+a^{-1}xv)U^*+a^{-1}vV^*.\]
        This means that the orbits are given by
        \[\mathcal{O}_+=\{K(a,x)V^*:(a,x)\in \mathrm{Aff}\}=\{uU^*+vV^*:v\in \mathbb{R},\, u>0\}\]
        and
        \[\mathcal{O}_-=\{K(a,x)(-V^*):(a,x)\in \mathrm{Aff}\}=\{uU^*+vV^*:v\in \mathbb{R},\, u<0\}.\]
        The affine group has the structure of semi-direct product between $R=(\mathbb{R}^+,\cdot)$ and $H=(\mathbb{R},+)$ with $\phi(a)x=a\cdot x$.
        Using Equation \ref{semi:representation}, we see that the orbit $\mathcal{O}_\pm$ with $F=\pm V^*$ gives the representations $\pi_\pm$.
\end{example}
\begin{example}[Shearlet Group]
        The shearlet group has two four-dimensional orbits.
        The dual basis is then denoted by $A^*,\,B^*,\, C^*$ and $D^*$.
        The co-adjoint map is given by \begin{align*}         
        K(a,s,x_1,x_2)(\alpha A^*+\beta B^*+\gamma C^*+\delta D^*)&=(\alpha+\frac{\beta s}{2\sqrt{a}}+\gamma\frac{2x_1-s x_2}{2a }+\frac{\delta x_2}{2\sqrt{a}})A^*\\&\qquad+(\frac{\beta}{\sqrt{a}} +\frac{\gamma x_2}{a})B^*+\frac{\gamma}{a}C^*-(\frac{\gamma s}{a}-\frac{\delta}{\sqrt{a}}) D^*.
        \end{align*}
        This means that the orbits are given by
        \[\mathcal{O}_+=\{K(x)C^*:x\in \mathbb{S}\}=\{\alpha A^*+\beta B^*+\gamma C^*+\delta D^*:\alpha, \beta, \delta\in \mathbb{R},\, \gamma>0\}\]
        and
        \[\mathcal{O}_-=\{K(x)(-C^*):x\in \mathbb{S}\}=\{\alpha A^*+\beta B^*+\gamma C^*+\delta D^*:\alpha, \beta, \delta\in \mathbb{R},\, \gamma<0\}.\]
Let $H=\exp(\mathrm{span}(D, C))$. We can describe the group as a semi-direct product: Define $\phi_{\cdot}:(\mathbb{R}^+\times \mathbb{R})\to \text{Aut}(\mathbb{R}^2)$ defined by $\phi_{(a,s)}=S_sA_a$. Then we can define $(\mathbb{R}^+\times \mathbb{R})\ltimes_\phi \mathbb{R}^2$. Choosing $F=\pm C^*$ then gives the induced representations $\pi_\pm$ from Example \ref{ex:shearlet}.
\end{example}
\begin{comment}
\begin{example}[Engel Group]
            Let $A,\, B,\, C$ and $D$ be a basis for the Lie algebra consisting of the matrices
        \[A=(1,0,0,0)_\mathfrak{e},\,B=(0,1,0,0)_\mathfrak{e},\,C=(0,0,1,0)_\mathfrak{e},\text{ and }D=(0,0,0,1)_\mathfrak{e}\]   
        The dual basis is then denoted by $A^*,\,B^*,\, C^*$ and $D^*$.
        The co-adjoint map is given by \[K(a,b,c,d)(\alpha A^*+\beta B^*+\gamma C^*+\delta D^*)=(\alpha-b\gamma+(ab-b-c)\delta)A^*+(\beta +a\gamma+\frac{a^2}{2}\delta)B^*+(\gamma +a\delta)C^*+\delta D^*.\]
        Notice that the Engel group have no four dimensional orbits, hence the highest dimension is $\mathrm{dim}(\mathcal{O}_F)=2$. This means that $\SF\not=\{e\}$. If we choose $F=\delta D^{*}$ then $H=\exp(\mathrm{Span}(B,C,D))$ and we get the representation on $L^2(\mathbb{R},\mathrm{d}a)$ \[(\pi_{\delta D^*}(a,b,c,d)f)(a')=e^{-2\pi i\delta(a'^2 b/2-a'c+d)}f(a'a).\]
\end{example}
\end{comment}

\subsection{Time-frequency analysis}\label{sec:prelim_TF}
Time-frequency analysis and its subfields are central pieces of applied harmonic analysis and have a strong connection to Weyl quantization. The (projective) representation $\pi(z) \psi(t) = \pi(x,\omega) \psi(t) = e^{2\pi i \omega t} \psi(t-x)$ used in Section \ref{sec:prelim_weyl} induces the main object of time-frequency analysis, the \emph{short-time Fourier transform}, defined as
\begin{align}\label{eq:stft_def}
    V_\phi \psi(z) = \langle \psi, \pi(z) \phi\rangle.    
\end{align}
Note that this quantity is a special case of a wavelet transform \eqref{eq:wavelet_def} if we disregard the $\pi(z)^*$ convention. %We expand on this connection below.
In applications, it is often the squared modulus $|V_\phi \psi|^2$ of the short-time Fourier transform, the \emph{spectrogram}, that is used because it is non-negative. In view of Theorem \ref{theorem:duflo-moore} it also integrates to $1$ and so can be seen as a probability distribution on $\R^{2d}$.

As time-frequency analysis is only tangentially related to the subject at hand, we only give a short primer on those topics which will be needed later on. A standard reference for time-frequency analysis in which much of the content of this section can be found is \cite{grochenig2013foundations}.

%\subsubsection{Wigner distribution}\label{sec:prelim_wigner}
We already saw the cross-Wigner distribution $W(\psi, \phi)$ in Section \ref{sec:prelim_weyl} as the dequantization of a rank-one operator. In time-frequency analysis, the cross-Wigner distribution is known as a time-frequency distribution related to short-time Fourier transform by% and an alternative to the short-time Fourier transform. In fact, they are related as
\begin{equation*}%\label{eq:wigner_is_stft}
    W(\psi, \phi)(x,\omega) = 2^d e^{4\pi i x \cdot \omega} V_{\check{\phi}} \psi(2x,2\omega).
\end{equation*}
The Wigner distribution also has an inversion formula and is uniquely determined by the window up to a constant, i.e., if $W(\psi) = W(\phi)$ then $\psi = c\phi$ where $|c| = 1$.

%\subsubsection{Abstract time-frequency analysis}\label{sec:prelim_abstract_TF}
The more general form of the short-time Fourier transform, the wavelet transform $\mathcal{W}_\phi$ mentioned in \eqref{eq:wavelet_def}, is the typical generalization of time-frequency analysis. To define is, we need a locally compact group $G$, a Hilbert space $\mathcal{H}$ and a square integrable irreducible representation $\pi : G \to \mathcal{U}(\mathcal{H})$. The most classical incarnation of this setup, apart from time-frequency analysis, is \emph{time-scale analysis} on the \emph{affine group}, discussed in Section \ref{sec:group_examples_prelim}. %The affine group is non-unimodular and consequently the Duflo-Moore operator $\D^{-1}$ is not the identity. 
In time-scale analysis, the squared modulus of the wavelet transform is called the \emph{scalogram} and is used similarly to the spectrogram. Other classical generalizations include the shearlet transform \cite{guo2006, Kutyniok2012} and various two-dimensional wavelet transforms such as the similitude transform \cite{Ali2000Wavelet, Antoine1996}.

\subsection{Quantum harmonic analysis}\label{sec:prelim_qha}
The theory of \emph{quantum harmonic analysis} (QHA), originally developed by R. Werner in 1984 \cite{werner1984}, lays out a framework in which many of the classical operations and results of harmonic analysis is generalized to operators. Specifically, convolutions are defined between functions on $\R^{2d}$ and operators on $L^2(\R^{d})$, as well as pairs of operators on $L^2(\R^d)$, as
\begin{align}\label{eq:func_op_op_op_weyl_def}
    f \star S = \int_{\R^{2d}} f(z) \pi(z) S \pi(z)^*\,dz,\qquad T \star S(z) = \tr\big(T \pi(z) PSP \pi(z)^*\big)
\end{align}
where $\pi(z) f(t) = \pi(x,\omega) f(t) = e^{2\pi i \omega \cdot t} f(t-x)$ is the projective representation of the Weyl-Heisenberg group from Section \ref{sec:prelim_TF} and $P : f \mapsto \check{f}$ is the parity operator. Together with the Fourier-Wigner transform $\FW : \mathcal{S}^1 \to C_0(\R^{2d})$, from Section \ref{sec:prelim_fw_fko}, defined as $\FW(S)(z) = e^{-i \pi x \cdot \omega} \tr(\pi(-z) S)$, they make up the main tools of QHA. Both convolution definitions in \eqref{eq:func_op_op_op_weyl_def} conjugate the operator $S$ by $\pi(z)$, this operation is called an \emph{operator translation} and is commonly denoted by $\alpha_z(S) = \pi(z) S \pi(z)^*$.

\begin{remark}
    Note that the above definitions are not consistent with the rest of this article but rather agree with those commonly found in the greater quantum harmonic analysis literature. When generalizing beyond the Weyl-Heisenberg group it is beneficial to adopt some alternative formulations. We will detail these and their motivation in Section \ref{sec:prelim_aqha} below.    
\end{remark}

\subsubsection{Standard properties}\label{sec:prelim_qha_key_props}
Below we collect some of the main properties of operator convolutions, all of which have counterparts in harmonic analysis. For proofs and more details, see \cite{Luef2018}.
\begin{proposition}\label{prop:op_conv_properties}
    Let $f,g \in L^1(\R^{2d}),\, S \in \mathcal{S}^p,\, T \in \mathcal{S}^q$ for $1 \leq p,q, r \leq \infty$ with $\frac{1}{p} + \frac{1}{q} = 1 + \frac{1}{r}$ and $R \in B(L^2)$ be compact, then
    \begin{enumerate}[label=(\roman*)]
        \item $f \star S$ is positive if $f$ is non-negative and $S$ is positive,
        \item \label{item:op_op_positive}$T \star S$ is non-negative if $T$ and $S$ are positive,
        \item $(f \star S)^* = \bar{f} \star S^*$,
        \item \label{item:ass1} $(f \star S) \star T = f * (S \star T)$,
        \item \label{item:ass2} $(f * g) \star S = f  \star (g \star S)$,
        \item \label{item:func_compact_conditon} $f \star R$ is compact,
        \item $\Vert f \star S \Vert_{\mathcal{S}^p} \leq \Vert f \Vert_{L^1} \Vert S \Vert_{\mathcal{S}^p}$,
        \item \label{item:op_op_conv_interpol} $\Vert T \star S \Vert_{L^r} \leq \Vert S \Vert_{\mathcal{S}^p} \Vert T \Vert_{\mathcal{S}^q}$.
    \end{enumerate}
\end{proposition}

To formulate Fourier-analytic results we need a Fourier transform which works on the phase space $\R^{2d}$. This will be the same symplectic Fourier transform as in Section \ref{sec:prelim_fw_fko}, namely
\begin{align}\label{eq:F_sigma_def}
    \mathcal{F}_\sigma(f)(z) = \int_{\R^{2d}} f(z') e^{-2\pi i \sigma(z,z')}\,dz',
\end{align}
where $z = (x,\omega),\, z' = (x', \omega')$ and $\sigma(z,z') = \omega x' - \omega' x$ is the standard symplectic form. The correct setting to investigate Fourier-related properties is requiring all operators to be of trace-class and here we have some more noteworthy results.
\begin{proposition}\label{prop:more_qha_props}
    Let $f \in L^1(\R^{2d})$ and $T, S \in \mathcal{S}^1$, then
    \begin{enumerate}[label=(\roman*)]
        \item \label{item:conv_theorem_func_op}$\FW(f \star S) = \mathcal{F}_\sigma(f) \FW(S)$,
        \item \label{item:conv_theorem_op_op}$\mathcal{F}_\sigma(T \star S) = \FW(T) \FW(S)$,
        \item $\tr(f \star S) = \tr(S) \int_{\R^{2d}} f(z)\,dz$,
        \item \label{item:op_op_conv_explicit}$\int_{\R^{2d}} T \star S(z)\,dz = \tr(T)\tr(S)$.
    \end{enumerate}
\end{proposition}
These results should be seen as generalizations and consequences of the standard convolution theorem. 

Part of the power of quantum harmonic analysis is that various seemingly unrelated objects can be realized as operator convolutions where their properties have clear explanations. Examples of this include localization operators and Cohen's class distributions from time-frequency analysis \cite{Luef2018, Luef2019} (Section \ref{sec:prelim_TF}) and Bergman-Fock Toeplitz operators \cite{luef2021jfa, Fulsche2023QHAFock}.

\subsubsection{Relation to Weyl quantization}\label{sec:prelim_qha_weyl_relation}
As quantum harmonic analysis is concerned with the interplay between functions and operators, it should come as no surprise that the theory is intimately connected with that of Weyl quantization. In fact, one can use Weyl quantization to induce all of the main operations of QHA by noting that
\begin{align}\label{eq:weyl_quant_qha_conseq}
    T\star S(z) = a_T * a_S (z),\qquad f \star S = A_{f * a_S}.
\end{align}
Moreover, Weyl quantization can be written explicitly as
\begin{align}\label{eq:weyl_quantization_classical}
    a_S = \mathcal{F}_\sigma (\FW(S))
\end{align}
when $S \in \mathcal{S}^2$, so the two central Fourier transforms of QHA make up the core of the Weyl quantization procedure. From \eqref{eq:weyl_quant_qha_conseq} and \eqref{eq:weyl_quantization_classical}, all of Proposition \ref{prop:more_qha_props} follow once we note that
\begin{align*}
    \int_{\R^{2d}} a_S(z)\,dz = \mathcal{F}_\sigma(a_S)(0) = \mathcal{F}_\sigma(\mathcal{F}_\sigma(\FW(S)))(0) = \FW(S)(0) = \tr(S).
\end{align*}
The action of translating an operator using the operator translation $\alpha_z : S \mapsto \pi(z)S\pi(z)^*$ which is used in quantum harmonic analysis can also be seen as being induced through Weyl quantization via the relation
\begin{align*}
    \alpha_z(S) = A_{T_z a_S}
\end{align*}
from Proposition \ref{prop:classic_quantization_props}, where $T_z f(x) = f(x-z)$ is the translation operator.

\subsubsection{Generalizing beyond the Weyl-Heisenberg group}\label{sec:prelim_aqha}
The setting of $L^2(\R^d)$ as the Hilbert space on which operators act and $\pi : \R^{2d} \to \mathcal{U}(L^2(\R^d))$ as the associated representation has underpinned our discussion of QHA so far. The work on generalizing this beyond the Weyl-Heisenberg group started in \cite{Berge2022} where $\R^{2d}$ was replaced by the affine group $\operatorname{Aff}$, the Hilbert space $L^2(\R^d)$ specialized to $L^2(\R^+)$ and the irreducible representation set to
\begin{align*}
U : \operatorname{Aff} \to \mathcal{U}(L^2(\R^+)),\qquad U(x,a) \psi(t) = e^{2\pi i x \cdot t} \psi(at).    
\end{align*}
In this setting, a suitable Wigner transform was also developed using the work in \cite{Berge2022affine} which required the introduction of a suitable replacement of the symplectic Fourier transform, namely the \emph{Fourier-Kirillov transform} which acts on $\operatorname{Aff}$.
Later on in \cite{Halvdansson2023}, using similar techniques, the setup was further generalized to allow for irreducible representations $\pi$ of (separable) locally compact groups $G$. In this setting and the affine, the function-operator and operator-operator convolutions take the following form with the right Haar measure convention:
\begin{align*}
    f \star S = \int_{G} f(x) \pi(x)^* S \pi(x)\dmr(x),\qquad T \star S(x) = \tr\big(T \pi(x)^* S \pi(x)\big).
\end{align*}
For these convolutions, variants of Proposition \ref{prop:op_conv_properties} and \ref{prop:more_qha_props} hold with the exception of the two Fourier properties. QHA was also generalized to the non-separable abelian setting in \cite{Fulsche2025}.

It turns out that most of the difficulties related to generalizing QHA stem from the possible non-unimodularity of the underlying group. Specifically, this leads to the introduction of the concept of \emph{admissible} operators which are operators $S$ on $\mathcal{H}$ such that $\D^{-1} S \D^{-1}$ is trace-class. One central result is that an operator-operator convolution only is integrable if one of the operators is trace-class and the other is admissible as is clear from the formula
\begin{align}\label{eq:qha_op_op_conv}
    \int_G T \star S(x) \dmr(x) = \tr(T) \tr(\D^{-1}S \D^{-1}).
\end{align}

\section{Two Fourier Transforms}\label{sec:main_weyl}
We are now ready to set up the Fourier-Wigner and Fourier-Kirillov transform alluded to in the introduction. Throughout the section we will assume that $\SF=\{e\}$ of $G$, and hence the co-adjoint orbit have the same dimension as the group. We will also assume that $H$ is a closed subgroup inducing the representation described in Section \ref{sec:prelim_rep_theory_exp_lie}. %Define the map  In this case, we have that the orbit $\mathcal{O}$ has a group structure given by 

\subsection{Fourier-Wigner transform}\label{sec:FW}
The Fourier-Wigner transform defined in this section is the Fourier transform in the Plancherel theorem, see e.g. \cite[Thm.~3.48]{fuhr2004}, restricted to the representation $\pi$. For convenience, we give the proofs of properties we will need in later sections.

Let $\pi_F:G\to \mathcal{H}$ be the representation associated to the point $F\in \mathfrak{g}^*$ discussed in Section \ref{sec:rep_co_adj_orb}, where $\mathrm{dim}(\mathcal{O}_F)=\mathrm{dim}(G)$. From now on, we will drop the $F$ subscript. %Then if $f\in L^1_l(G)$, one can define the integrated representation by
%\begin{align*}
%    \pi(f)=\int_Gf(x)\pi(x) \dml(x),   
%\end{align*}
%see \cite[Sec.~3.2]{folland2016course} for well-definedness of this definition.

\begin{definition}[Fourier-Wigner transform]\label{def:fourier_wigner}
    Let $A\in \mathcal{S}^2$ be such that $A\D$ extends to a trace class operator. We then define the \emph{Fourier-Wigner transform} $\FW(A)$ as
    \begin{align*}
        \FW(A)(x)=\tr(A\D\pi(x)).
    \end{align*}
    If $f \in L^1_r(G)$ and $\phi\in \mathrm{Dom}(\D)$, then we define the \emph{inverse Fourier-Wigner transform} by
\begin{align*}
    \FW^{-1}(f)(\phi)=(\pi(\check{f}) \circ \D )\phi = \int_G f(x)\pi(x^{-1})\D\phi\dmr(x).   
\end{align*}
\end{definition}
We will now show that the Fourier-Wigner transform extends to an injective isometry from $\mathcal{S}^2(\mathcal{H})\to L^2_r(G)$, with the additional property that the inverse Fourier-Wigner transform is a left inverse. To do so, we first need a preliminary result.

\begin{lemma}\label{lemma:dense_subspace_xd}
    The space $\mathcal{S}^1\D^{-1}\cap \mathcal{S}^2$ is a dense subspace of $\mathcal{S}^2$.
\end{lemma}
\begin{proof}
    Let $A = \sum_n a_n (\psi_n \otimes \phi_n)$ be an arbitrary operator in $\mathcal{S}^2$, fix $\varepsilon > 0$ and define
    \begin{align*}
        A_N = \sum_{n=1}^N a_n (\psi_n \otimes \phi_n),\qquad X_N = \sum_{n=1}^N a_n(\psi_n \otimes \D^{-1} \xi_n)
    \end{align*}
    where $\xi_n = \D \phi_n'$ and $\phi_n'$ are functions in the domain of $\D$ which satisfy $\Vert \phi_n - \phi_n' \Vert < \varepsilon$ for each $n$. This is possible because the domain of $\D$ is dense in $\mathcal{H}$ by Theorem \ref{theorem:duflo-moore}. Since $A \in \mathcal{S}^2$, there exists a $N$ such that $\Vert A - A_N \Vert_{\mathcal{S}^2} < \varepsilon$. We now claim that $\Vert A - X_N \Vert_{\mathcal{S}^2} < \varepsilon(1 + \Vert A \Vert_{\mathcal{S}^2})$, which can be made arbitrarily small, and that $X_N \D \in \mathcal{S}^1$. Indeed, $X_N \D \in \mathcal{S}^1$ since
    \begin{align*}
        X_N \D = \sum_{n=1}^N a_n (\psi_n \otimes \xi_n)
    \end{align*}
    is a finite sum of bounded rank-one operators and $\psi_n, \xi_n \in \mathcal{H}$. Moreover, $\Vert X_N - A_N \Vert_{\mathcal{S}^2}$ can be bounded as
    \begin{align*}
        \Vert X_N - A_N \Vert_{\mathcal{S}^2} &\leq \left( \sum_{n=1}^N |a_n|^2 \Vert \psi_n \Vert \big\Vert \D^{-1} \xi_n - \phi_n \big\Vert \right)^{1/2} = \left( \sum_{n=1}^N |a_n|^2 \big\Vert \phi_n' - \phi_n \big\Vert \right)^{1/2}\\
        &\leq \varepsilon \Vert A_N \Vert_{\mathcal{S}^2} \leq \varepsilon \Vert A \Vert_{\mathcal{S}^2},
    \end{align*}
    finishing the proof.
\end{proof}

\begin{proposition}\label{prop:fourier_wigner_rank_one}
    Let $\psi \in \mathcal{H}$ and $\phi \in \operatorname{Dom}(\D^{-1})$. Then
    \begin{align*}
        \FW(\psi \otimes \D^{-1} \phi)(x) = \mathcal{W}_\phi \psi(x).
    \end{align*}
\end{proposition}
\begin{proof}
    Plugging in the definition of $\FW$, we find that
    \begin{align*}
        \FW(\psi \otimes \D^{-1} \phi)(x) &= \FW((\psi \otimes \phi) \D^{-1})(x)\\
        &= \tr\big( (\psi \otimes \phi) \pi(x) \big)\\
        &= \langle \psi, \pi(x)^* \phi\rangle.
    \end{align*}
\end{proof}
Using linearity, we can use this last result to define $\FW$ on all of $\mathcal{S}^1\D^{-1}$ using the singular value decomposition. In the same way, we will now show how $\FW$ can be further extended to all of $\mathcal{S}^2$ using Lemma \ref{lemma:dense_subspace_xd}.

\begin{proposition}\label{prop:FW_isometry}
    The Fourier-Wigner transform $\FW : \mathcal{S}^1\D^{-1} \to L^2_r(G)$ can be continuously extended to all of $\mathcal{S}^2$ as an injective isometry meaning that
    $$
    \Vert A \Vert_{\mathcal{S}^2} = \Vert \FW(A) \Vert_{L^2_r}
    $$
    for all $A \in \mathcal{S}^2$.
\end{proposition}
\begin{proof}
    We will show that $\FW$ is an isometry on $\mathcal{S}^1\D^{-1}\cap \mathcal{S}^2$ which is dense in $\mathcal{S}^2$ by Lemma \ref{lemma:dense_subspace_xd}. Let $A \in \mathcal{S}^1 \D^{-1}$ so that we can write
    \begin{align*}
        A = \sum_n s_n (\psi_n \otimes \phi_n)
    \end{align*}
    where $\phi_n \in \operatorname{Dom}(\D)$ and both $(\psi_n)_n, (\phi_n)_n$ are orthonormal sequences in $\mathcal{H}$. We can now use Proposition \ref{prop:fourier_wigner_rank_one} to conclude that,
    \begin{align*}
        \Vert \FW(A) \Vert_{L^2_r}^2 = \big\langle \FW(A), \FW(A) \big\rangle &= \sum_n \sum_m s_n \overline{s_m} \big\langle \FW(\psi_n \otimes \phi_n), \FW(\psi_m \otimes \phi_m) \big\rangle\\
        &=\sum_n \sum_m s_n \overline{s_m} \big\langle \mathcal{W}_{\D \phi_n} \psi_n, \mathcal{W}_{\D \phi_m} \psi_m \big\rangle\\
        &=\sum_n \sum_m s_n \overline{s_m} \langle \psi_n, \psi_m \rangle \overline{\langle \phi_n, \phi_m \rangle} \\
        &=\sum_n |s_n|^2 \Vert \phi_n \Vert^2 \Vert \psi_n \Vert^2 = \sum_n |s_n|^2 = \Vert A \Vert_{\mathcal{S}^2}^2
    \end{align*}
    where we made use of the Duflo-Moore theorem (Theorem \ref{theorem:duflo-moore}). %A standard polarization argument, which we omit for brevity, now yields that $\FW$ is unitary.
\end{proof}
The inverse Fourier-Wigner transform $\FW^{-1}$ is a left inverse to $\FW$ as we show in this next proposition. %We show that the reverse composition is in fact a projection in Theorem \ref{thm:projection}. 
\begin{proposition}\label{prop:FW_identity}
    The composition $\FW^{-1} \circ \FW$ is the identity operator on $\mathcal{S}^2$.
\end{proposition}
\begin{proof}
We will prove the proposition for $\mathcal{S}^1 \D^{-1}$ and conclude the full case by the density from Lemma \ref{lemma:dense_subspace_xd}. Let $A, B \in \mathcal{S}^1 \D^{-1}$ be arbitrary, then with $A = \sum_n a_n (\psi_n \otimes \D^{-1}\phi_n)$ and $B = \sum_n b_n (\xi_n \otimes \D^{-1} \eta_n)$, we have that
\begin{align}\label{eq:FWFW_inv_full}
    \big\langle \FW^{-1}(\FW(A)), B \big\rangle_{\mathcal{S}^2} &= \sum_n \sum_m a_n \overline{b_m} \big\langle \FW^{-1}(\mathcal{W}_{\phi_n}\psi_n), (\xi_m \otimes \D^{-1}\eta_m) \big\rangle_{\mathcal{S}^2}
\end{align}
by Proposition \ref{prop:fourier_wigner_rank_one}. We can now compute each term as
\begin{align*}
    \big\langle \FW^{-1}(\mathcal{W}_{\phi}\psi), (\xi \otimes \D^{-1}\eta) \big\rangle_{\mathcal{S}^2} &= \tr\big( \FW^{-1}(\mathcal{W}_\phi \psi) (\D^{-1}\eta \otimes \xi) \big)\\
    &= \sum_n \left\langle \FW^{-1}(\mathcal{W}_\phi \psi)\langle e_n, \xi \rangle \D^{-1}\eta, e_n  \right\rangle\\
    &= \left\langle \FW^{-1}(\mathcal{W}_\phi\psi) \D^{-1}\eta, \xi \right\rangle\\
    &= \int_G \langle \psi, \pi(x) \phi \rangle \langle \pi(x) \D \D^{-1} \eta, \xi \rangle\dml(x)\\
    &= \int_G \mathcal{W}_\phi \psi(x^{-1}) \overline{\mathcal{W}_\eta \xi(x^{-1})}\dml(x)\\
    &=\int_G \mathcal{W}_\phi \psi(x) \overline{\mathcal{W}_\eta \xi(x)}\dmr(x) = \langle \psi, \xi \rangle \overline{\langle \D^{-1}\phi, \D^{-1}\eta \rangle}
\end{align*}
using Theorem \ref{theorem:duflo-moore}. Plugging this result into each term of \eqref{eq:FWFW_inv_full}, we conclude that
\begin{align*}
    \big\langle \FW^{-1}(\FW(A)), B \big\rangle_{\mathcal{S}^2} = \sum_n \sum_m a_n \overline{b_m}\langle \psi_n, \eta_m \rangle \overline{\langle \D^{-1} \phi_n, \D^{-1}\eta_m \rangle} = \langle A, B \rangle_{\mathcal{S}^2}
\end{align*}
which implies that $\FW^{-1} \circ \FW$ is the identity on $\mathcal{S}^1\D^{-1}$.% By the density result mentioned above, this relation holds on all of $\mathcal{S}^2$.
\end{proof}

The following is as close as we are able to get to a Riemann-Lebesgue lemma for $\FW$ as decay at infinity is not always possible to define for a group.
\begin{proposition}\label{prop:fw_riemann_lebesgue}
    If $A \in \mathcal{S}^1\D^{-1}$, then $\FW(A)$ is continuous. Moreover, if for all $\varepsilon > 0$ and $\psi, \phi \in \mathcal{H}$ there exist sets $E(\varepsilon, \psi, \phi)$ such that
    $$
    x \in E(\varepsilon, \psi, \phi) \implies |\mathcal{W}_\phi \psi(x)| < \varepsilon,
    $$
    then given an $\varepsilon > 0$, there exists a set $E(\varepsilon, A)$ which is a finite intersection of sets of the form $E(\varepsilon', \psi, \phi)$ such that
    $$
    x \in E(\varepsilon, A) \implies |\FW(A)(x)| < \varepsilon.
    $$
\end{proposition}
\begin{proof}
    For continuity, let $x \to x_0$ and estimate
    \begin{align*}
        \big| \FW(A)(x) - \FW(A)(x_0) \big| &= \big|\tr\big(A \D \pi(x) - A \D \pi(x_0)\big)\big|\\
        &\leq \Vert A\D \Vert_{\mathcal{S}^1} \Vert \pi(x) - \pi(x_0) \Vert_{\mathcal{S}^\infty}
    \end{align*}
    using $\Vert AB \Vert_{\mathcal{S}^1} \leq \Vert A \Vert_{\mathcal{S}^1} \Vert B \Vert_{\mathcal{S}^\infty}$. The last factor goes to zero by the strong continuity of the representation.

    For the decay, fix $\varepsilon > 0$ and decompose $A\D$ as $A\D = \sum_{n} a_n (\psi_n \otimes \phi_n)$ where $(\psi_n)_n$ and $(\phi_n)_n$ are orthonormal sequences. Since $(a_n)_n \in \ell^1$, we can find an integer $N > 0$ such that $\sum_{n=N+1}^\infty |a_n| < \varepsilon/2$ and therefore,
    \begin{align*}
        \big| \FW(A)(x) \big| \leq \sum_{n = 1}^N |a_n| \big| \mathcal{W}_{\phi_n} \psi_n(x)\big| + \underbrace{\sum_{n=N+1}^\infty |a_n| \big| \mathcal{W}_{\phi_n}\psi_n(x)\big|}_{< \varepsilon/2}
    \end{align*}
    since $|\mathcal{W}_{\phi_n} \psi_n(x)| \leq 1$ by Cauchy-Schwarz. Now let $E(\varepsilon, A) = \bigcap_{n=1}^N E\Big(\frac{\varepsilon}{2 \Vert A \Vert_{\mathcal{S}^1}}, \psi_n, \phi_n\Big)$ and the result follows.
\end{proof}

\begin{remark}
    The decay condition on the wavelet transform $\mathcal{W}_\phi \psi$ is rather natural. In the time-frequency literature, one often places assumptions on the short-time Fourier transform of the type
    \begin{align*}
        |V_\phi \psi(x)| \leq C(1+|x|)^{-s}
    \end{align*}
    for some positive $s$ dependent on the dimension of the ambient space.
\end{remark}

In computations involving representations on semi-direct products, we can compute the Fourier-Wigner transform using the integral kernel of an operator. Denote by $R=H/G$. Recall that if $A\in \mathcal{S}^2(L^2_r(R))$ then the integral kernel $K_A\in L^2(R\times R)$ is defined by \[A\phi=\int_{R}K_A(\cdot,t)\phi(t)\mathrm{d}\mu^{R}(t).\]
\begin{proposition}\label{semi:simple}
Assume that $G=R\rtimes_\phi H$ and $A=\psi_1\otimes\psi_2\in \mathcal{S}^2(L^2_r(R))$, where $\psi_2\in \mathrm{Dom}( \D)$. 
    Then the kernel of $A\D \pi((r,h))$ is given by 
    \[\psi_1(s)\overline{\psi_2(tr^{-1})}\frac{\sqrt{|\det_H(\phi_*(tr^{-1}))|}}{\sqrt{\Delta_R(tr^{-1})}}e^{-2\pi i F\log_H(\phi(tr^{-1})h)}.\] Taking the trace gives 
    \[\mathcal{F}_W(\psi_1\otimes\psi_2)(r,h)=\int_R  \psi_1(t)\overline{\psi_2(tr^{-1})}e^{-2\pi i F\log_H(\phi(tr^{-1})h)}\frac{\sqrt{|\det_H(\phi_*(tr^{-1}))|}}{\sqrt{\Delta_R(tr^{-1})}}\mathrm{d}\mu^{R}(t).\]
\end{proposition}
\begin{proof}
    Using \eqref{semi:representation} we have
\begin{align*}
   (\D\pi(r,h)\xi)(t)&= \frac{\sqrt{|\det_H(\phi_*(t))|}}{\sqrt{\Delta_R(t)}}e^{-2\pi i F\log_H(\phi(t)h)}\xi(tr).
\end{align*}
    The definition of the Fourier-Wigner transform now implies that
    \begin{align*}
        \FW(\psi_1\otimes \psi_2)(r)&=\int_R\psi_1(s)\overline{\psi_2(tr^{-1})}\frac{\sqrt{|\det_H(\phi_*(tr^{-1}))|}}{\sqrt{\Delta_R(tr^{-1})}}e^{-2\pi i F\log_H(\phi(tr^{-1})h)}\xi(t) \dm^R(t)\\
    \end{align*}
    Hence the kernel is 
    \[\psi_1(s)\overline{\psi_2(tr^{-1})}\frac{\sqrt{|\det_H(\phi_*(tr^{-1}))|}}{\sqrt{\Delta(tr^{-1})}}e^{-2\pi i F\log_H(\phi(t^{-1}rtr^{-1})h)}.\]
    Using that the trace of a rank-one operator $\psi_1\otimes\psi_2$ is given by 
    \begin{align*}
        \tr(\psi_1\otimes\psi_2)=\int_R\psi_1(t)\overline{\psi_2(t)} \dm^R(t)
    \end{align*}
    gives the result.
\end{proof}
Lastly we collect some fundamental results on how the inverse Fourier-Wigner transform interacts with translations, adjoints and convolutions.
\begin{lemma}\label{lemma:wigner_rep_interaction}
    Let $f,g \in L^1_r(G)$ and $x \in G$, then
    \begin{enumerate}[label=(\roman*)]
        \item \label{item:wigner_rep_interaction}$\FW^{-1}(f)\pi(x) = \FW^{-1}\left( \frac{1}{\sqrt{\Delta(x)}} L_{x^{-1}} f \right)$,
        \item \label{item:rep_wigner_interaction}$\pi(x) \FW^{-1}(f) = \FW^{-1}(R_x f)$,
%        \item \label{item:dual integral} $\FW^{-1}(f)^* = \FW^{-1}(\Delta^{1/2}\check{f})$,
%        \item \label{item:convolution} $\FW^{-1}(f\ast h) = \FW^{-1}(f) \circ \D^{-1} \circ \FW^{-1}(h) = \pi(\check{f}) \circ \FW^{-1}(h)$.
    \end{enumerate}
\end{lemma}
\begin{proof}
    Properties \ref{item:wigner_rep_interaction} and \ref{item:rep_wigner_interaction} are essentially reformulations of \cite[Thm.~3.9]{folland2016course}. For \ref{item:wigner_rep_interaction}, we can compute
    \begin{align*}
        \FW^{-1}(f) \pi(x) &= \int_G \check{f}(y) \pi(y) \D\pi(x)\dml(y) \\
        &= \int_G \check{f}(y) \pi(y x) \sqrt{\Delta(x)} \D\dml(y) && \Big(\D \pi(x) = \sqrt{\Delta(x)}\pi(x) \D\Big)\\
        &= \int_G \check{f}(zx^{-1}) \pi(z) \sqrt{\Delta(x)} \D\dml(zx^{-1}) && \Big(z = yx \implies y = zx^{-1}\Big)\\
        &= \int_G \check{f}(zx^{-1}) \pi(z) \frac{1}{\sqrt{\Delta(x)}} \D\dml(z) && \Big(d\mu_l(zx^{-1}) = \Delta(x^{-1}) d\mu_l(z)\Big)\\
        &= \int_G \widecheck{L_x f}(z) \pi(z) \frac{1}{\sqrt{\Delta(x)}} \D \dml(z) && \Big( \check{f}(zx^{-1}) = \widecheck{L_x f}(z) \Big)\\
        &= \FW^{-1}\left(\frac{1}{\sqrt{\Delta(x)}} L_{x^{-1}} f\right).
    \end{align*}
    To see that $\check{f}(zx^{-1}) = \widecheck{L_{x^{-1}} f}(z)$, note that $\check{f}(zx^{-1}) = f(xz^{-1})=(L_{x^{-1}}f)(z^{-1})$.

    Meanwhile for \ref{item:rep_wigner_interaction}, we have that
    \begin{align*}
        \pi(x) \FW^{-1}(f) &= \int_G \check{f}(y) \pi(xy) \D\dml(y)\\
        &= \int_G \check{f}(x^{-1}z) \pi(z) \D\dml(x^{-1}z) &&\Big(z = xy \implies y = x^{-1}z \Big)\\
        &= \int_G \widecheck{R_x f}(z) \pi(z) \D \dml(x^{-1}z) &&\Big( \check{f}(x^{-1}z) = \widecheck{R_x f}(z) \Big)\\
        &= \int_G \widecheck{R_x f}(z) \pi(z) \D \dml(z) &&\Big( d\mu_l(yz) = d\mu_l(z) \Big)\\
        &= \FW^{-1}(R_x f). \qedhere
    \end{align*}
%    Item \ref{item:dual integral} can be verified by computing a weak action. Let $\phi,\xi\in \mathcal{H}$, then 
%    \begin{align*}
%        \langle \FW^{-1}(f)\phi, \xi\rangle_\mathcal{H}&=\int_G f(x)\langle\pi(x^{-1})\D\phi,\xi\rangle\dmr(x)\\
%        &=\int_G f(x)\langle\phi,\D\pi(x)\xi\rangle\dmr(x)\\
%        &=\int_G f(x)\langle\phi,\sqrt{\Delta(x)}\pi(x)\D\xi\rangle\dmr(x)\\
%        &=\left\langle \phi,\int_G f(x)\sqrt{\Delta(x)}\pi(x)\D\xi\dmr(x) \right\rangle.
%    \end{align*}
%    For \ref{item:convolution} we have from that \cite[Thm.~3.11]{folland2016course} that $\pi(\check{f}\ast \check{g})=\pi(\check{f})\pi(\check{g})$ which we can use to deduce that
%    \begin{align*}
%      \FW^{-1}(f\ast g) = \pi(\check{f}\ast \check{g})  \D= \pi(\check{f})\D^{-1}\D\pi(\check{g})\D= \FW^{-1}(f) \circ \D^{-1} \circ \FW^{-1}(g).  
%    \end{align*}
\end{proof}

\subsection{Fourier-Kirillov transform}\label{sec:FKO}
We will now construct a function Fourier transform. Notice in Section \ref{sec:orbit} the symplectic form is defined via group action. Hence map $\kappa $ induces, ut to a constant, an isomorphism from $L^2_r(G)$ to $L^2(\mathcal{O}_F, \mathrm{d}\omega)$ by mapping $f\mapsto f\circ\kappa^{-1}$. Since the right Haar measure is unique up to a constant, we will choose the right Haar measure of the group $G$ that makes this map into an isomorphism.% \todo{Should this last arrow be a $\mapsto$?}
\begin{lemma}\label{lem:symplectic form}
    Let $X_1,\ldots,X_{2n}$ be a basis for $\mathfrak{g}$, and $\mathrm{d}Y_1,\ldots, \mathrm{d}Y_{2n}$ be the dual basis. Set $\mathrm{Pf}_F$ to be the Pfaffian \[\mathrm{Pf}_F=\mathrm{Pf}(F([X_i ,X_j]))=\frac{1}{2^{n}n!}\sum_{\sigma\in \mathrm{S}_{2n}}\mathrm{sgn}(\sigma)\prod_{j=1}^n F([X_{\sigma(2j-1)},X_{\sigma(2j)}]),\]
    where $\mathrm{S}_{2n}$ is the symmetric group. Then 
    \[\mathrm{Pf}_F\cdot \Delta(\kappa^{-1}(Y)) \mathrm{d} Y_1\wedge\cdots\wedge \mathrm{d} Y_{2n}=\mathrm{Pf}_F\cdot \Delta(\kappa^{-1}(Y))\,\mathrm{d}Y = \mathrm{d}\omega(Y)=\frac{\omega^n(Y)}{n!}.\]
\end{lemma}
\begin{proof}
    By the definition of the symplectic form, see \cite[Lem.~4.1.3]{ArnalDidier2020RoSL}, we have that 
    \[\mathrm{Pf}(Y[X_{\sigma(2j-1)},X_{\sigma(2i)}]) \mathrm{d}Y_1\wedge\cdots\wedge \mathrm{d}Y_{2n} = \mathrm{d}\omega(Y).\]
    By the definition of $\kappa$, we can write 
    \[Y=K(\kappa^{-1}(Y)^{-1})F,\]
    hence
    \[\mathrm{Pf}(Y[X_{\sigma(2i-1)},X_{\sigma(2i)}])\,\mathrm{d}Y=\mathrm{Pf}(F([\Ad_{\kappa^{{-1}}(Y)}X_{\sigma(2i-1)},\Ad_{\kappa^{-1}(Y)}X_{\sigma(2i)}]))\,\mathrm{d}Y.\]
    The Pfaffian, see e.g. \cite{Pfaffian}, has the property that  \[\mathrm{Pf}(BAB^*)=\det(B)\mathrm{Pf}(A).\]
    Using this together with \eqref{eq:modular_function_adjoint}
    we get that 
     \[\mathrm{Pf}_F\cdot\Delta(\kappa^{-1}(Y)) \mathrm{d}Y_1\wedge\cdots\wedge \mathrm{d}Y_{2n}=\mathrm{d}\omega(Y).\qedhere\]
\end{proof}

Define the \textit{Fourier-Kirillov transform} to be a mapping $\FKO:L^2_r(G)\to L^2_r(G)$ by 
\begin{align}\label{eq:FKO_def}
    \FKO(f)(x)&=\frac{1}{\sqrt{|\mathrm{Pf}_F|\cdot\Delta(x)}}\int_{G}f(y)e^{2\pi i\langle \kappa(x), \log(y)\rangle}\frac{1}{\sqrt{\Theta(\log(y))}}\dmr(y)\\\nonumber
    &= \frac{1}{\sqrt{|\mathrm{Pf}_F|\cdot\Delta(x)}}\int_{\mathfrak{g}}f(\exp(X))e^{2\pi i\langle\kappa(x), X\rangle}\sqrt{\Theta(X)}\,\mathrm{d}X.
\end{align}
The \textit{inverse Fourier-Kirillov transform} $\FKO^{-1}:L^2_r(G)\to L^2_r(G)$ is given by 
\begin{align}\nonumber
    \FKO^{-1}(f)(x)&=\sqrt{\Theta(\log(x))}\int_{\mathfrak{g}^*}\chi_{\mathcal{O}_F}(Y)f(\kappa^{-1}(Y))e^{-2\pi i\langle Y, \log(x)\rangle}\sqrt{|\mathrm{Pf}_F|\cdot\Delta(\kappa^{-1}(Y))}\, \mathrm{d}Y\\\label{eq:inverse_FKO}
    &=\sqrt{\Theta(\log(x))}\int_{G}f(y)e^{-2\pi i\langle \kappa(y), \log(x)\rangle}\frac{1}{ \sqrt{|\mathrm{Pf}_F|\cdot\Delta(y)}} \dmr(y).
\end{align}
\begin{proposition}\label{prop:FKO_unitary}
    The Fourier-Kirillov and inverse Fourier-Kirillov transforms are well-defined and have the property that 
    \begin{align*}
        \FKO \FKO^{-1}=\mathrm{id}.
    \end{align*}
    Moreover, $\FKO^{-1}$ is unitary, i.e.,
    \begin{align*}
        \big\langle \FKO^{-1}(f), \FKO^{-1}(g)\big\rangle_{L^2_r}=\langle f, g\rangle_{L^2_r},
    \end{align*}
    and $\FKO^{-1}\FKO$ is a projection onto $\FKO^{-1}(L^2_r(G))$.
\end{proposition}
\begin{proof}
Let us start by showing well-definedness of the transforms.
It is clear that the following maps are surjective isomorphisms:
\begin{itemize}
    \item $\mathrm{Mul}_{\sqrt{\Theta}}:L^2_r(G)\to L^2(\mathfrak{g},\mathrm{d}X)$ 
defined by \[\mathrm{Mul}_{\sqrt{\Theta}}(f)(X)=\sqrt{\Theta(X)}f(e^X).\]
    \item $\mathcal{F}:L^2(\mathfrak{g},\mathrm{d}X)\to L^2(\mathfrak{g}^*,\mathrm{d}Y)$ 
defined by \[\mathcal{F}(f)=\int_\mathfrak{g}f(X)e^{2\pi i \langle X, Y\rangle}\,\mathrm{d}X.\]
\item $\mathrm{Mul}_{1/\sqrt{|\mathrm{Pf}_F|\Delta}}:L^2(\mathcal{O}_F, \mathrm{d}Y)\to  L^2(\mathcal{O}_F, \mathrm{d}\omega)$ 
defined by $\mathrm{Mul}_{1/\sqrt{|\mathrm{Pf}_F|\Delta}}(f)(Y)=\frac{f(Y)}{\sqrt{|\mathrm{Pf}_F|\cdot\Delta(\kappa^{-1}(Y))}}$. See Lemma \ref{lem:symplectic form}.
\end{itemize}
Additionally, we have the restriction operator $\mathrm{Res}_{\mathcal{O}_F}:L^2(\mathfrak{g}^*,\mathrm{d}Y)\to L^2(\mathcal{O}_F,\mathrm{d}Y)$ defined by \[\mathrm{Res}_{\mathcal{O}_F}(f)(Y)=\left.f\right|_{\mathcal{O}_F}(Y)\]
and the inclusion  operator $\mathrm{Inj}_{\mathcal{O}_F}:L^2(\mathcal{O}_F,\mathrm{d}Y)\to L^2(\mathfrak{g}^*,\mathrm{d}Y)$ defined by \[\mathrm{Inj}_{\mathcal{O}_F}(f)(Y)=\begin{cases}f(Y)&Y\in \mathcal{O}_F,\\
0&\text{otherwise}.\end{cases}\]
Now we can write 
\[(\FKO f)(x)=\mathrm{Mul}_{1/\sqrt{|\mathrm{Pf}_F|\Delta}}\circ\mathrm{Res}_{\mathcal{O}_F}\circ\mathcal{F}\circ\mathrm{Mul}_{\sqrt{\Theta}}(f)(\kappa(x))\]
and 
\[(\FKO^{-1}f)(x)=\mathrm{Mul}_{\sqrt{\Theta}}^{-1}\circ\mathcal{F}^{-1}\circ\mathrm{Inj}_{\mathcal{O}_F}\circ\mathrm{Mul}_{1/\sqrt{|\mathrm{Pf}_F|\Delta}}^{-1}(f)(x).\]
Notice that $\mathrm{Inj}_{\mathcal{O}_F}\circ\mathrm{Res}_{\mathcal{O}_F}$ is a projection and $\mathrm{Res}_{\mathcal{O}_F}\circ\mathrm{Inj}_{\mathcal{O}_F}=\mathrm{id}$ which implies that $\FKO^{-1} \FKO$ is a projection and $\FKO \FKO^{-1} = \operatorname{id}$.

That the operator $\mathcal{F}^{-1}_W$ is unitary follows from the computation 
\begin{align*}
    \big\langle \FKO^{-1}(f), \FKO^{-1}(g)\big\rangle_{L^2_r}
    &=\big\langle \mathrm{Inj}_{\mathcal{O}_F}(\mathrm{Mul}_{1/\sqrt{|\mathrm{Pf}_F|\Delta}}^{-1}(f)), \mathrm{Inj}_{\mathcal{O}_F}(\mathrm{Mul}_{1/\sqrt{|\mathrm{Pf}_F|\Delta}}^{-1}(g))\big\rangle_{L^2(\mathfrak{g}^*, \mathrm{d}Y)}\\
    &=\big\langle \mathrm{Mul}_{1/\sqrt{|\mathrm{Pf}_F|\Delta}}^{-1}(f), \mathrm{Mul}_{1/\sqrt{|\mathrm{Pf}_F|\Delta}}^{-1}(g)\big\rangle_{L^2(\mathcal{O}_F, \mathrm{d}Y)}\\
    &=\langle f, g\rangle_{L^2_r},
\end{align*}
where we have used that the inclusion operator is an injective isomorphism.
\end{proof}

\begin{lemma}\label{lemma:kirillov_complex_conjugate_final}
    Let $f\in L^2_r(G)$, then $\mathcal{F}_{\mathrm{KO}}(\sqrt{\Delta_G(\cdot)}\check{f}(\cdot))=\overline{\mathcal{F}_{\mathrm{KO}}(\overline{f(\cdot)})}$.
\end{lemma}
\begin{proof}
Using \eqref{eq:modular quotion} we have that
    \begin{align*}
    \overline{\FKO(\overline{f})(x)}
     &=\frac{1}{\sqrt{|\mathrm{Pf}_F|\Delta(x)}}\int_\mathfrak{g}f(\exp(X))e^{-2\pi i\langle \kappa(x),X\rangle}\sqrt{\Theta(X)}\,\mathrm{d}X\\    
     &=\frac{1}{\sqrt{|\mathrm{Pf}_F|\Delta(x)}}\int_\mathfrak{g}f(\exp(-X))e^{2\pi i\langle \kappa(x),X\rangle}\sqrt{\Theta(-X)}\,\mathrm{d}X\\    
     &=\frac{1}{\sqrt{|\mathrm{Pf}_F|\Delta(x)}}\int_\mathfrak{g}f((\exp(X))^{-1})e^{2\pi i\langle \kappa(x),X\rangle}\sqrt{\Theta(-X)}\,\mathrm{d}X\\    
     &=\frac{1}{\sqrt{|\mathrm{Pf}_F|\Delta(x)}}\int_\mathfrak{g}\sqrt{\Delta_G(\exp(X))}\check{f}(\exp(X))e^{2\pi i\langle \kappa(x),X\rangle}\sqrt{\Theta(X)}\,\mathrm{d}X,
    \end{align*}
    which completes the proof.
\end{proof}

\begin{proposition}\label{prop:translate_of_FKO}
    Let $f\in L^2_r(G)$ and define $L_x f(y)=f(x^{-1}y)$, $R_xf(y)=f(yx)$ be the right and left translation operators, and $\Psi_{x^{-1}} = L_x R_x$. Then,
    \begin{equation*}
        \FKO^{-1}(R_x f) = \sqrt{\Delta(x)} \Psi_{x^{-1}} \FKO^{-1}(f) \quad \text{and} \quad R_{x}\mathcal{F}_{\mathrm{KO}}(f)=\sqrt{\Delta(x)}\mathcal{F}_{\mathrm{KO}}(\Psi_{x}f).
    \end{equation*}
\end{proposition}
\begin{proof}
    We have that 
    \begin{align*}
        \big\langle K((yx^{-1})^{-1}) F, \log(z) \big\rangle &= \big\langle F, \Ad_{yx^{-1}}\log(z) \big\rangle\\
        &= \big\langle K(y^{-1}) F, \Ad_{x^{-1}}\log(z) \big\rangle =\big\langle K(y^{-1}) F, \log(\Psi_{x^{-1}}z) \big\rangle.
    \end{align*}
   Hence 
    \begin{align*}
    \FKO^{-1}(f)(x^{-1}zx)
    &=\sqrt{\Theta(\log(x^{-1}zx))}\int_{G}f(y)e^{-2\pi i\langle \kappa(y), \log(\Psi_{x^{-1}}(z))\rangle}\frac{1 }{\sqrt{|\mathrm{Pf}_F|\cdot\Delta(y)}} \dmr(y)\\
    &=\sqrt{\Theta(\Ad_{x^{-1}}\log(z))}\int_{G}f(y)e^{-2\pi i\langle \kappa(yx^{-1}), \log(z)\rangle}\frac{1}{ \sqrt{|\mathrm{Pf}_F|\cdot\Delta(y)}} \dmr(y)\\
    &=\sqrt{\Theta(\Ad_{x^{-1}}\log(z))}\int_{G}f(yx)e^{-2\pi i\langle \kappa(y), \log(z)\rangle}\frac{1}{ \sqrt{|\mathrm{Pf}_F|\cdot\Delta(yx)}} \dmr(y)\\
    &=\frac{\sqrt{\Theta(\log(z))}}{\sqrt{\Delta(x)}}\int_{G}
    f(yx)e^{-2\pi i\langle \kappa(y), \log(z)\rangle}\frac{1}{ \sqrt{|\mathrm{Pf}_F|\cdot\Delta(y)}} \dmr(y).
    \end{align*}
    The final step follows from \eqref{eq:volume exponential} together with the determinant being invariant under automorphisms. The identity $R_{x}\mathcal{F}_{\mathrm{KO}}(f)=\sqrt{\Delta(x)}\mathcal{F}_{\mathrm{KO}}(\Psi_{x}f)$ is proved similarly.
\end{proof}

For semi-direct products we have a slight simplification.
\begin{proposition}\label{semi:simple2}
    Let $G=R\rtimes_\phi H$. Then 
    \begin{multline*}
\FKO(f)(r_1,h_1)=\\ \frac{1}{\sqrt{|\mathrm{Pf}_F|\cdot\Delta_G(r_1,h_1)}}\int_{G}f(r_2,h_2)e^{2\pi i\langle F,\log_G(r_1r_2r_1^{-1},h_1\phi(r_1^{-1}r_2^{-1})(h_2)h_1^{-1})\rangle}\frac{\mathrm{d}\mu_r(r_2,h_2)}{\sqrt{\Theta(\log(r_2,h_2))}} .
    \end{multline*}
\end{proposition}
\begin{proof}
Let $(r_1,h_1), \, (r_2,h_2)\in R\rtimes_\phi H$.
We have that 
   \begin{align*}
       \langle K((r_1,h_1)^{-1})F, \log_G(r_2,h_2)\rangle
       &=\langle F, \Ad_{(r_1,h_1)}\log_G(r_2,h_2)\rangle\\
        &=\langle F,\log_G((r_1,h_1)(r_2,h_2)(r_1,h_1)^{-1})\rangle\\
        &=\langle F,\log_G(r_1r_2r_1^{-1},h_1\phi(r_1^{-1}r_2^{-1})(h_2)h_1^{-1}))\rangle.
    \end{align*}
\end{proof}

\subsection{Combining the transforms}
As described in the introduction our goal is to define quantization by combining the Fourier transforms. Hence this section will be devoted to showing the connections between the range of the Fourier-Kirillov transform and the inverse Fourier-Wigner transform.
\begin{lemma}\label{lemma:invert_FKO_and_FW}
    Let $f$ be continuous on $G$ with compact support. Then
    \begin{align*}
        \FKO^{-1}(\FKO(f))(e) = \FW(\FW^{-1}(f))(e),
    \end{align*}
    where $e\in G$ is the group identity element.
\end{lemma}
\begin{proof}
This is simply a reformulation of  \cite[Prop.~6.3.1]{ArnalDidier2020RoSL} by writing
\[\FW(\FW^{-1}(f))(e)=\tr(\pi(\check{f})\circ \mathcal{D}\circ \mathcal{D})=\tr(\mathcal{D}\circ \pi(\check{f})\circ \mathcal{D})\] and 
\begin{align*}
\FKO^{-1}(\FKO(f))(e)%&= \int_{\mathcal{O}_F}\int_\mathfrak{g}f(\exp(X))e^{2\pi i\langle Y,X\rangle} \sqrt{\Theta(X)}\,\mathrm{d}X\,\mathrm{d}Y\\
&= \int_{\mathcal{O}_F}\int_\mathfrak{g}\check{f}(\exp(X))e^{-2\pi i\langle Y,X\rangle}
\sqrt{\Theta(-X)}\,\mathrm{d}X\,\mathrm{d}Y. \qedhere
\end{align*}
\end{proof}

\begin{theorem}\label{thm:projection}
   Let $f\in L^2_r(G)$. Then
   \begin{align}\label{eq:fkofkofwfw}
        \FKO^{-1}(\FKO(f))=\FW(\FW^{-1}(f)).    
   \end{align}
   Consequently, $\FW \FW^{-1}$ is a projection onto $\FKO^{-1}(L^2_r(G))$ and $\FKO^{-1}(L^2_r(G))=\FW(\mathcal{S}^2(\mathcal{H}))$.
\end{theorem}
\begin{proof}
    We will start by assuming that $f$ is continuous with compact support. Then from Lemma \ref{lemma:invert_FKO_and_FW} we have that \[\FKO^{-1}(\FKO(f))(e)=\FW(\FW^{-1}(f))(e).\] Using Proposition \ref{prop:translate_of_FKO} we have that
    \begin{align*}
        \FKO^{-1}(\FKO(f))(x) &=R_x \FKO^{-1}(\FKO(f))(e)\\
        &=\FKO^{-1}(\sqrt{\Delta(x)}\Psi_x\FKO(f))(e)\\
        &=\FKO^{-1}(\FKO(R_x f))(e).
    \end{align*}
    Additionally, using that $R_x \FW(A) = \FW(\pi(x) A)$ (from the definition of $\FW$) and Lemma \ref{lemma:wigner_rep_interaction} \ref{item:rep_wigner_interaction} we can deduce that
    \begin{align*}
        \FW(\FW^{-1}(f))(x)&=R_x \FW(\FW^{-1}(f))(e)\\
        &=\FW(\pi(x)\FW^{-1}(f))(e)\\
        &=\FW(\FW^{-1}(R_xf))(e).
    \end{align*}
    Hence we have that 
    \begin{align*}
        \FKO^{-1}(\FKO(f))=\FW(\FW^{-1}(f))    
    \end{align*}
    for continuous $f$ with compact support. Since continuous functions with compact support are dense in $L^2_r(G)$, the result follows.
    
    To see that $\FW \FW^{-1}$ is a projection onto $\FKO^{-1}(L^2_r(G))$, let $F = \FKO^{-1}(f)$ be an arbitrary elements of $\FKO^{-1}(L^2_r(G))$. Then
    \begin{align*}
        \FW (\FW^{-1}(F)) = \FKO^{-1}(\FKO(F)) = \FKO^{-1}(\FKO(\FKO^{-1}(f))) = \FKO^{-1}(f) = F
    \end{align*}
    by \eqref{eq:fkofkofwfw} and Proposition \ref{prop:FKO_unitary}.

    Lastly for $\FKO^{-1}(L^2_r(G))=\FW(\mathcal{S}^2(\mathcal{H}))$, we first let $g \in \FKO^{-1}(L^2_r(G))$ with $g = \FKO^{-1}(f)$ for some $f \in L^2_r(G)$. Then
    \begin{align*}
        g = \FKO^{-1}(f) = \FKO^{-1}(\FKO(\FKO^{-1}(f))) = \FW(\FW^{-1}(g)) \in \FW(\mathcal{S}^2)
    \end{align*}
    where we used Proposition \ref{prop:FKO_unitary} for the second step. Similarly, for the other direction we let $g \in \FW(\mathcal{S}^2)$ with $g = \FW(A)$ for some $A \in \mathcal{S}^2$. It then holds that
    \begin{align*}
        g = \FW(A) = \FW(\FW^{-1}(\FW(A))) = \FKO^{-1}(\FKO(g)) \in \FKO^{-1}(L^2_r(G))
    \end{align*}
    and consequently the sets are identical.
\end{proof}

\begin{example}[Affine Group]
    Denote by $X_1=U$ and $X_2=V$. % the basis for $\mathfrak{aff}$ consisting of \[X_1=\begin{bmatrix}
%        1&0\\0&0
%    \end{bmatrix},\qquad X_2=\begin{bmatrix}
%        0&1\\0&0
%    \end{bmatrix}. \]
%    Then $[X_1,X_2]=X_2$. We will denote the dual basis by $Y_1,Y_2$.
%    Let us compute $\FKO$ in the case that $F=(0,1)=Y_2$. It is also well known that $\Theta(x_1X_1+x_2X_2)=\frac{x_1e^{x_1}}{e^{x_1} - 1}$
    We have for $F=V^*$ that
    \[\mathrm{Pf}_{V^*}=\frac{1}{2}\sum_{\sigma\in S_2}\mathrm{sgn}(\sigma)V^*([X_{\sigma(1)},X_{\sigma(2)}])=1.\]
    Additionally, we have that $\kappa(a,x)=aV^*-xU^*$, since co-adjoint map is given by \[K(a,x)(V^*)=a^{-1}V^*+a^{-1}xU^*.\] 
    This means that \[\langle\kappa(a,x), uU+vV \rangle =av-xu,\] which is the standard symplectic form.
    Hence 
    \begin{align*}
        \FKO(f)(a,x) &= \sqrt{a}\int_{\mathrm{Aff}}f(b,y)e^{2\pi i(ay/\lambda(\log(b))-x\log(b))}\frac{\dmr(b,y)}{\sqrt{b\lambda(\log(b))}}\\
        &=\sqrt{a}\int_{\mathbb{R}^{2}}f\left(e^u,v\cdot \lambda(u)\right)e^{2\pi i (av-xu)} \sqrt{e^u\lambda(u)}\, \mathrm{d}u \, \mathrm{d}v.
    \end{align*}

    Notice that if we had chosen the other orbit corresponding to $F=-V^*$ we would get the Fourier transformation 
    \begin{align*}
        \FKO^-(f)(a,x)&=\sqrt{a}\int_{\mathrm{Aff}}f(b,y)e^{2\pi i(x\log(b)-ay/\lambda(\log(b)))}\frac{\dmr(b,y)}{\sqrt{b\lambda(\log(b))}} \\
        &=\sqrt{a}\int_{\mathbb{R}^{2}}f\left(e^u,v\lambda(u)\right)e^{2\pi i (xu-av)}\sqrt{e^u\lambda(u)} \, \mathrm{d}u \, \mathrm{d}v.
    \end{align*}
\end{example}
\begin{example}[Shearlet Group]
    For the Shearlet group we have that $\D \phi(a,t)=a\phi(a,t)$.
    Let \[Af(c,r)=\int_{\mathbb{R}^{+}\times \mathbb{R}}K_A((c,r),(b,t))f(b,t)\,\frac{\mathrm{d}b\,\mathrm{d}t}{b}.\]
   % We have that 
    %\[A\D\pi(a,s,x_1,x_2)f(c,r)=\int_{\mathbb{R}^{+}\times \mathbb{R}}K_A\left((c,r),\left(\frac{b}{a},t-s\sqrt{\frac{b}{a}}\right)\right)e^{2\pi i (\frac{b}{a}x_1+\sqrt{\frac{b}{a}}tx_2-s\frac{b}{a}x_2)}\frac{b}{a}f(b,t)\,\frac{\mathrm{d}b\,\mathrm{d}t}{b}.\]
    Hence we have that 
    \[\FW(A)(a,s,x_1,x_2)=\int_{\mathbb{R}^{+}\times \mathbb{R}}K_A((b,t),(ab,t+s\sqrt{b}))e^{-2\pi i (bx_1+\sqrt{b}tx_2)}\,\mathrm{d}b\,\mathrm{d}t.\]

    Denote the basis by $X_1=A,\, X_2=B,\, X_3=C$, and $X_4=D$, and $Y_i$ the dual basis. Computing Pfaffian for $F=\pm Y_3$ we get that $|\mathrm{Pf}_{\pm Y_3^*}|=1$. Computing for $F=Y_3^*$ we get that \[\langle \kappa(a,s,x_1,x_2), \alpha A+\sigma B+\xi_1 C+\xi_2 D)\rangle=\frac{sx_2}{2}\alpha+\sqrt{a}(s\xi_2-x_2\sigma) + a\xi_1-x_1\alpha.\]
    The Fourier-Kirillov transform for $F=Y_3^*$ hence becomes
    \begin{multline*}\label{eq:FKO_def}
        \FKO(f)(a,s,x_1,x_2)%&=a\int_{G}f(\tilde{a},\tilde{b},\tilde{c},\tilde{d})e^{2\pi i\left(\frac{sx_2-2x_1}{2}\log(\tilde{a})+a\frac{\tilde{c}-\tilde{b}\tilde{d}/2}{\log(\tilde{a})}+\frac{d\tilde{b}-x_2\sqrt{a}\tilde{d}}{\lambda(\log(\tilde{a}/2))}\right)}\frac{\mathrm{d}\tilde{a}\,\mathrm{d}\tilde{b}\,\mathrm{d}\tilde{c}\,\mathrm{d}\tilde{d}}{\tilde{a}\sqrt{\lambda(\log(\tilde{a}))}\lambda(\log(\tilde{a})/2)}\\\nonumber
        \\= a\int_{\mathfrak{g}}f(\exp(\alpha,\beta,\gamma,\delta))e^{2\pi i\left(\frac{sx_2}{2}\alpha+\sqrt{a}(s\xi_2-x_2\sigma) + a\xi_1-x_1\alpha\right)}\sqrt{\lambda(\alpha)}\lambda(\alpha/2)\,\mathrm{d}\alpha\,\mathrm{d}\sigma\,\mathrm{d}\xi_1\,\mathrm{d}\xi_2,
    \end{multline*}
    where \[\exp(\alpha,\beta,\gamma,\delta)=(e^\alpha,\sigma \lambda(\alpha/2),\xi_1\lambda(\alpha)+\sigma \xi_2\lambda(\alpha/2)^2/2,\xi_2\lambda(\alpha/2)).\]
    
\end{example}

\section{Weyl Quantization}\label{sec:quantization_properties}
So far, we have extensively discussed the Fourier-Wigner (Section \ref{sec:FW}) and Fourier-Kirillov (Section \ref{sec:FKO}) transforms. We will now compose these to define our version of Weyl quantization in the same way as was done in Section \ref{sec:prelim_fw_fko}. This means that we define the quantization mapping as
\begin{align}\label{eq:quantization_anchor}
    A : f \mapsto A_f = \FW^{-1}(\FKO^{-1}(f)).
\end{align}
From the results about the Fourier-Wigner and Fourier-Kirillov transforms of Section \ref{sec:main_weyl}, we can deduce the following important property.
\begin{theorem}\label{theorem:quantization_main}
    The quantization mapping $A : L^2_r(G) \to \mathcal{S}^2$ and its inverse $a : \mathcal{S}^2 \to L^2_r(G)$ are both linear unitary isometries.
\end{theorem}
\begin{proof}
    Linearity of $A$ and $a$ follow from the definitions of $\FW$ and $\FKO$. For the unitary isometry part, note first that for $A, B \in \mathcal{S}^2$,
    \begin{align*}
        \big\langle \FKO(\FW(A)), \FKO(\FW(B)) \big\rangle_{L^2_r} = \big\langle \FKO^{-1}(\FKO(\FW(A))), \FKO^{-1}(\FKO(\FW(B))) \big\rangle_{L^2_r}
    \end{align*}
    since $\FKO^{-1}$ is unitary by Proposition \ref{prop:FKO_unitary}. The same proposition also states that $\FKO^{-1} \FKO$ is a projection onto $\FKO^{-1}(L^2_r)$ which is equal to $\FW(\mathcal{S}^2)$ by Theorem \ref{thm:projection}. Consequently,
    \begin{align*}
        \big\langle \FKO^{-1}(\FKO(\FW(A))), \FKO^{-1}(\FKO(\FW(B))) \big\rangle_{L^2_r} = \big\langle \FW(A), \FW(B) \big\rangle_{L^2_r}
    \end{align*}
    and the result now follows from $\FW$ being unitary by Proposition \ref{prop:FW_isometry}.

    The other direction of the quantization follows from a similar argument. Let $f, g \in L^2_r(G)$ and note that by $\FW$ being unitary (Proposition \ref{prop:FW_isometry}) and by $\FW \FW^{-1}$ being a projection onto $\FKO^{-1}(L^2_r(G))$ (Theorem \ref{thm:projection}),
    \begin{align*}
        \big\langle \FW^{-1}(\FKO^{-1}(f)), \FW^{-1}(\FKO^{-1}(g)) \big\rangle_{\mathcal{S}^2} &= \big\langle \FW(\FW^{-1}(\FKO^{-1}(f))), \FW(\FW^{-1}(\FKO^{-1}(g))) \big\rangle_{L^2_r}\\
        &= \big\langle \FKO^{-1}(f), \FKO^{-1}(g) \big\rangle_{L^2_r}\\
        &= \langle f, g \rangle_{L^2_r}
    \end{align*}
    where we in the last step used that $\FKO^{-1}$ is unitary by Proposition \ref{prop:FKO_unitary}.
\end{proof}

We now move to showing more detailed properties of the quantization mapping, starting with the actions of translation and conjugation.
\begin{example}[Affine group]
For the representation given by $F=V^*$ we get the quantization outlined in \cite{gayral2007fourier}. The dequantization of an operator $A$ is explicitly given by 
\begin{equation*}
    f_A(a,x) = \int_{-\infty}^{\infty}K_A\left(\frac{aue^{u}}{e^{u} - 1},\frac{au}{e^{u}-1}\right)e^{-2\pi i x u} \, du,
\end{equation*}
where $K_A \colon \mathbb{R}_{+} \times \mathbb{R}_{+} \to \mathbb{C}$ is the integral kernel of $A$ defined by \[A\psi(r) = \int_{0}^{\infty}K_A(r,s)\psi(s) \, \frac{ds}{s}, \qquad \psi \in L^{2}(\mathbb{R}_{+}).\] 
\end{example}
\begin{example}[Shearlet group]
    Let \[Af(c,r)=\int_{\mathbb{R}^{+}\times \mathbb{R}}K_A((c,r),(b,t))f(b,t)\,\frac{\mathrm{d}b\,\mathrm{d}t}{b}.\]
    Using the formula for the Fourier-Kirillov and Fourier-Wigner transform for the Shearlet group we get that the dequantization $f_A = \FKO(\mathcal{F}_W(A))$ of $A$ at the point $(a,b,c,d) \in \mathbb{S}$ has the explicit expression
\begin{multline*}
\int_{\mathbb{R}^2}K_A\left(\left( \frac{ae^\alpha}{\lambda(\alpha)},\frac{2\lambda(\alpha)s+\sigma \lambda(\alpha/2)^2} {2\sqrt{\lambda(\alpha)}\lambda(\alpha/2)}\right),\left(\frac{a}{\lambda(\alpha)},\,\frac{2\lambda(\alpha)s-\sigma \lambda(\alpha/2)^2} {2\sqrt{\lambda(\alpha)}\lambda(\alpha/2)}\right)\right)         e^{2\pi i\left(\frac{sx_2-2x_1}{2}\alpha-x_2\sigma\right)}\,\mathrm{d}\alpha\,\mathrm{d}\sigma.
\end{multline*}
\end{example}
\subsection{Translation and conjugation}
In this section we investigate how quantization behaves with respect to translation and complex conjugation. In the proofs, we will have to follow the quantization through the two Fourier transforms $\FW$ and $\FKO$ in \eqref{eq:quantization_anchor}. For translation, most of the work was done in Lemma \ref{lemma:wigner_rep_interaction} and Proposition \ref{prop:translate_of_FKO}. Note that the following relation is the same as was shown for the affine group in \cite[p.~47]{Berge2022}.
\begin{proposition}\label{prop:quantization_translation}
Let $f \in L^2_r(G)$ and $x \in G$, then
    $$
    \pi(x)^* A_f \pi(x) = A_{R_{x^{-1}}f}
    $$
    where $R_{x^{-1}} f(y) = f(yx^{-1})$.
\end{proposition}
\begin{proof}
    Assume that $f \in L^2_r(G) \cap L^1_r(G)$ so that we may apply Lemma \ref{lemma:wigner_rep_interaction}. The full result then follows by density and the continuity of $A$ from $L^2_r(G)$ to $\mathcal{S}^2$. We can then compute
    \begin{align*}
        \pi(x)^* A_f \pi(x) &= \pi(x^{-1}) \FW^{-1} (\FKO^{-1}(f)) \pi(x)\\
        &= \FW^{-1} ( R_{x^{-1}} \FKO^{-1}(f)) \pi(x)\\
        &= \FW^{-1} \left( \frac{1}{\sqrt{\Delta(x)}} L_{x^{-1}} R_{x^{-1}} \FKO^{-1}(f)\right)
    \end{align*}    
    where we in the last step used Proposition \ref{prop:translate_of_FKO} as $\frac{1}{\sqrt{\Delta(x)}} L_{x^{-1}} R_{x^{-1}} \FKO^{-1}(f) = \FKO^{-1}(R_{x^{-1}} f)$.
\end{proof}
We now move on to the quantization of the complex conjugates.
\begin{proposition}\label{prop:quantization_adjoint}
    Let $f \in L^2_r(G)$, then
    $$
    A_f^* = A_{\bar{f}}.
    $$
\end{proposition}
\begin{proof}
    To start, we assume that $f$ is such that $A_f\D \in \mathcal{S}^1$ so that we can apply the Fourier-Wigner transform. Recall that $A_f = \FW^{-1}(\FKO^{-1}(f))$. We will compute $\FW(A_f^*)$ and show that it coincides with $\FW(A_{\bar{f}}) = \FKO^{-1}(\bar{f})$ which yields the desired conclusion by the injectivity of $\FW$ guaranteed by Proposition \ref{prop:FW_isometry}.

    Let $A_f = \sum_m s_m (\psi_m \otimes \phi_m)$ be the singular value decomposition of $A_f$ and $A = A_f \D$ so that we may conclude that
    \begin{align*}
        \FW(A_f)(x) &= \FW(A\D^{-1})(x) = \tr\big(A\pi(x)\big) = \tr\big(A_f \D \pi(x)\big)\\
        &= \sum_n \langle A_f \D \pi(x)\xi_n, \xi_n \rangle\\
        &= \sum_n \left\langle \sum_m s_m (\psi_m \otimes \phi_m) \D \pi(x) \xi_n, \xi_n \right\rangle\\
        &= \sum_n \sum_m s_m \langle \D \pi(x) \xi_n, \phi_m \rangle \langle \psi_m, \xi_n\rangle\\
        &= \sum_m s_m \langle \psi_m, \pi(x^{-1}) \D \phi_m \rangle.
    \end{align*}
    Meanwhile, using that $A_f^* = \D^{-1} A^*$ we find that
    \begin{align*}
        \FW(A_f^*)(x) &= \FW(\D^{-1} A^*)(x) = \FW(\D^{-1} A^* \D \D^{-1})\\
        &= \tr\big(\D^{-1} A^* \D \pi(x)\big) = \tr\big(A_f^* \D \pi(x)\big)\\
        &= \sum_n \langle A_f^* \D \pi(x) \xi_n, \xi_n \rangle\\
        &= \overline{ \sum_m s_m \langle \pi(x)^* \D \psi_m, \phi_m\rangle }
    \end{align*}
    where we in the last step used the same argument as for $\FW(A_f)$. Continuing, we find that
    \begin{align*}
        \FW(A_f^*)(x) &= \overline{ \sum_m s_m \langle \psi_m, \sqrt{\Delta(x)}\pi(x)\D \phi_m\rangle }\\
        &= \sqrt{\Delta(x)}\, \overline{ \sum_m s_m \langle \psi_m, \pi(x)\D \phi_m\rangle }\\
        &= \sqrt{\Delta(x)}\, \overline{\FW(A_f)(x^{-1})}.
    \end{align*}
    From Lemma \ref{lemma:kirillov_complex_conjugate_final} we have that
    $$
    \FKO^{-1}(\bar{f})(x) = \sqrt{\Delta(x)}\overline{\widecheck{\FKO^{-1}(f)}(x)}
    $$
    and plugging this into the above yields
    \begin{align*}
        \FW(A_f^*)(x) = \sqrt{\Delta(x)} \frac{1}{\sqrt{\Delta(x)}} \FKO^{-1}(\bar{f})(x)
    \end{align*}
    from which it follows that $A_f^* = A_{\bar{f}}$ using the injectivity of $\FW$.

    For the full case, for any $\varepsilon > 0$ we can by Lemma \ref{lemma:dense_subspace_xd} find an $A_g \in \mathcal{S}^2$ with $A_g\D \in \mathcal{S}^1$ such that 
    $$
    \Vert A_f - A_g \Vert_{\mathcal{S}^2} < \varepsilon.
    $$
    It then holds that $A_g^* = A_{\bar{g}}$ and so we have
    \begin{align*}
        \Vert A_f^* - A_{\bar{f}} \Vert_{\mathcal{S}^2} \leq \Vert A_f^* - A_g^* \Vert_{\mathcal{S}^2} + \Vert A_{\bar{g}} - A_{\bar{f}} \Vert_{\mathcal{S}^2} < 2\varepsilon.
    \end{align*}
    Since $\varepsilon$ was arbitrary, the conclusion follows.
\end{proof}

\subsection{Algebraic structure}\label{sec:algebraic_structure}
With the basic properties of the quantization mapping established, we now move on to endowing each step of the quantization map with algebraic structure. First we show that $A : L^2_r(G) \to \mathcal{S}^2$ is a form of isomorphism and then move on to the individual Fourier transforms.

\subsubsection{An $H^*$-algebra $\ast$-isomorphism}\label{sec:main_iso}
In his seminal 1966 paper \cite{Pool1966}, Pool showed that Weyl quantization on the Weyl-Heisenberg group can be realized as an isometric $\ast$-isomorphism between $H^*$-algebras. In this section, we set out to do the same in our more general setup. As a first step, we recall the definition of an $H^*$-algebra \cite{Wong1998}. Note that the Hilbert space $\mathcal{H}$ in the definition below is not the same as the one associated with our representation $\pi$.
\begin{definition}\label{def:H_star_alg}
    Let $\mathcal{H}$ be a complex and separable Hilbert space and $\cdot : a, b \mapsto a \cdot b, {}^* : a \mapsto a^*$ two operations satisfying the properties
    \begin{enumerate}[label=(\roman*)]
        \item $a \cdot (b+c) = a \cdot b+a \cdot c,\quad (a+b) \cdot c = a \cdot c+b \cdot c$,
        \item $\lambda(a \cdot b) = (\lambda a)\cdot b = a\cdot(\lambda b)$,
        \item $a \cdot (b \cdot c) = (a \cdot b) \cdot c$,\label{item:associativity}
        \item $a^{**} = a$,
        \item $(a+b)^* = a^* + b^*$,\label{item:involution_distributive}
        \item $(a \cdot b)^* = b^* \cdot a^*$,\label{item:involution_of_product}
        \item $(\lambda a)^* = \bar{\lambda} a^*$,
        \item $\Vert a^* \Vert = \Vert a \Vert$,
        \item $\Vert a \cdot b \Vert \leq \Vert a \Vert \Vert b \Vert$, \label{item:prod_submultiplicativity}
        \item $\langle a \cdot b,c \rangle = \langle b, a^* \cdot c \rangle$, \label{item:H_adjoint}
    \end{enumerate}
    for $a, b, c \in \mathcal{H}$ and $\lambda \in \mathbb{C}$. Then we call $(\mathcal{H},\, \cdot,\, {}^*)$ an $H^*$-algebra with respect to the given operations.
\end{definition}
We want to show that the quantization mapping $A : f \mapsto A_f$ \eqref{eq:quantization_anchor} provides a $\ast$-isomorphism between the function and operator spaces. Since we have already showed that complex conjugation corresponds to taking the adjoint in Proposition \ref{prop:quantization_adjoint}, the main obstacle to this is defining a mapping on $L^2_r(G)$ which should correspond to composing operators through quantization. In the Weyl-Heisenberg case, this need is fulfilled by \emph{twisted multiplication} \cite[Chap.~2.3]{folland1989harmonic}, sometimes also referred to as \emph{Moyal products} \cite[Sec.~13.3.3]{hall2013quantum}, which have an explicit definition. In the general setting, such an explicit expression is too much to hope for but it is possible to set up a suitable mapping in a backwards manner as
\begin{align}\label{eq:twisted_multiplication}
    f \sharp g := a_{A_f A_g} \iff A_{f \sharp g} = A_f A_g
\end{align}
and we refer to it as twisted multiplication as it coincides with the Weyl-Heisenberg definition. Note that the twisted multiplication of two $L^2_r(G)$ functions is another $L^2_r(G)$ function since
\begin{align*}
    \Vert A_{f \sharp g} \Vert_{\mathcal{S}^1} = \Vert A_f A_g \Vert_{\mathcal{S}^1} \leq \Vert A_f \Vert_{\mathcal{S}^2} \Vert A_g \Vert_{\mathcal{S}^2}
\end{align*}
which implies that $A_{f \sharp g} \in \mathcal{S}^2$ since $\mathcal{S}^1 \subset \mathcal{S}^2$ and in turn $f \sharp g \in L^2_r(G)$ by Theorem \ref{theorem:quantization_main}.

This manner of setting up the multiplication by tracing the quantization is not unique, see e.g. \cite{Bieliavsky2021} for a similar construction for a Kohn-Nirenberg type quantization. Combining twisted convolutions with the result of Proposition \ref{prop:quantization_adjoint}, we have all the ingredients we need to set up the desired isomorphism. However, we must first verify that the two spaces indeed are $H^*$-algebras in the first place.
\begin{lemma}\label{lemma:H*_HS}
    The Hilbert space $\mathcal{S}^2$ equipped with the operations of composition and taking the adjoint,
    \begin{align*}
        A,B \,\,\mapsto\,\, A \circ B, \qquad A \,\,\mapsto\,\, A^*,
    \end{align*}
    is an $H^*$-algebra.
\end{lemma}
\begin{proof}
    All of the conditions of Definition \ref{def:H_star_alg} are standard properties of Hilbert-Schmidt operators with the possible exception of submultiplicativity, item \ref{item:prod_submultiplicativity}, which follows from $\Vert TS \Vert_{\mathcal{S}^2} \leq \Vert T \Vert_{\mathcal{S}^2} \Vert S \Vert_\infty$ and $\Vert S \Vert_\infty \leq \Vert S \Vert_{\mathcal{S}^2}$, and adjoints in inner products which is equivalent to $\tr(ABC^*) = \tr(BC^*A)$ and follows from trace cyclicity.
\end{proof}
\begin{lemma}\label{lemma:H*_L2}
    The Hilbert space $L^2_r(G)$ equipped with the operations of twisted multiplication and complex conjugation,
    \begin{align*}
        f,g \,\,\mapsto\,\, f \sharp g, \qquad f \,\,\mapsto\,\, \overline{f},
    \end{align*}
    is an $H^*$-algebra.
\end{lemma}
\begin{proof}
    We verify the associativity property \ref{item:associativity} for the twisted multiplication as
    \begin{align*}
        f \sharp (g \sharp h) = a_{A_f A_{g \sharp h}} = a_{A_f A_g A_h} = a_{A_{f \sharp g} A_h} = (f \sharp g) \sharp h.
    \end{align*}
    Meanwhile submultiplicativity \ref{item:prod_submultiplicativity} follows from
    \begin{align*}
        \Vert f \sharp g \Vert_{L^2_r} = \Vert A_{f \sharp g} \Vert_{\mathcal{S}^2} = \Vert A_f A_g \Vert_{\mathcal{S}^2} \leq \Vert A_f \Vert_{\mathcal{S}^2} \Vert A_g \Vert_{\mathcal{S}^2} = \Vert f \Vert_{L^2_r} \Vert g \Vert_{L^2_r}.
    \end{align*}
    Lastly for property \ref{item:H_adjoint} we use that the quantization mapping is unitary, that $\mathcal{S}^2$ is an $H^*$-algebra, Proposition \ref{prop:quantization_adjoint} and equation \eqref{eq:twisted_multiplication} to verify that
    \begin{align*}
        \langle f \sharp g, h \rangle_{L^2_r} = \langle A_f A_g, A_h \rangle_{\mathcal{S}^2} = \langle A_g, A_f^* A_h \rangle_{\mathcal{S}^2} = \langle A_g, A_{\bar{f}} A_h \rangle_{\mathcal{S}^2} = \langle g, \bar{f} \sharp h \rangle_{L^2_r}.
    \end{align*}
    The remaining conditions are easy to verify and are skipped in the interest of brevity.
\end{proof}
We are now essentially done with the work of showing that the mapping is a $\ast$-isomorphism.
\begin{theorem}
    The quantization mapping $A : L^2_r(G) \to \mathcal{S}^2$ as defined by \eqref{eq:quantization_anchor} is an isometric $\ast$-isomorphism of the $H^*$-algebra $(L^2_r(G), \sharp, \overline{ \phantom{a} })$ onto the $H^*$-algebra $(\mathcal{S}^2, \circ, {}^*)$.
\end{theorem}
\begin{proof}
    The mapping being an isometry follows from Proposition \ref{prop:FW_isometry} and Proposition \ref{prop:FKO_unitary}. That $A_f^* = A_{\overline{f}}$ was established in Proposition \ref{prop:quantization_adjoint} and $A_{f \sharp g} = A_f A_g$ is the contents of \eqref{eq:twisted_multiplication}.
\end{proof}

\subsubsection{Intermediate twisted $\ast$-isomorphisms}
While the twisted multiplication in \eqref{eq:twisted_multiplication} was crucial to setting up the $H^*$-algebra $\ast$-isomorphism, there is more structure induced by it available to uncover. Specifically, just as one can use the classical Fourier convolution theorem to define convolutions from multiplications on $L^2(\mathbb{R})$ as $f * g = \mathcal{F}^{-1}(\mathcal{F}(f) \cdot \mathcal{F}(g))$, we can also define an associated \emph{twisted convolution}. The same construction has also been made in the Weyl-Heisenberg setting and again there is an explicit expression for the operation available which we will be unable to obtain in this more general setting. Denoting the twisted convolution by $\natural$, generalizing the convolution theorem means that we need to satisfy
\begin{align}\label{eq:twisted_conv_from_mul}
    \FKO(f \natural g) = \FKO(f) \sharp \FKO(g) \implies f \natural g = \FKO^{-1}\big( \FKO(f) \sharp \FKO(g) \big)
\end{align}
since $\FKO$ is the Fourier transform in our setting and we use $\sharp$ as our multiplication. More explicitly, using $A_f = \FW^{-1}(\FKO^{-1}(f))$ and the definition of twisted convolution, we get that
\begin{align}\nonumber
    f \natural g &= \FKO^{-1}\big[ a_{A_{\FKO(f)} A_{\FKO(g)} } \big]\\\nonumber
    &=\FKO^{-1}\big[ a_{\FW^{-1}(f) \FW^{-1}(g) } \big]\\\nonumber
    &=\FKO^{-1}\big[ \FKO(\FW(  \FW^{-1}(f) \circ \FW^{-1}(g)  )) \big]\\\label{eq:twisted_conv_from_comp}
    &= \FW\big( \FW^{-1}(f) \circ \FW^{-1}(g) \big).
\end{align}
In view of the implication $\FW^{-1}(f \natural g) = \FW^{-1}(f) \circ \FW^{-1}(g)$, we could have obtained the same expression for $\natural$ by simply requiring that the triple $(\FW^{-1},\, \natural,\, \circ)$ should have the same relation as the standard Fourier triple $(\mathcal{F},\, *,\, \cdot)$ gets from the convolution theorem. From this point of view, it becomes clear that our original definition of $\sharp$ was induced by a similar relation but with $(A, \sharp, \circ)$ as the triple where $A : f \mapsto A_f$ is the quantization mapping.

So far we been endowing our quantization procedure $A$ with additional structure. In this section we finish this task by showing that the diagram in Figure \ref{fig:H_cd} commutes.
\begin{figure}[H]
    \centering
    \begin{tikzcd}
        \big(\mathcal{S}^2(\mathcal{H}),\, \circ,\,{}^*\big) \arrow[d, "\FW"] \arrow[rd, dotted, "a"]&\\
        \big(\FW(\mathcal{S}^2),\, \natural,\,\sqrt{\Delta(\cdot)}\, \overline{ \check{\phantom{a}} }\big) \arrow[r,"\FKO"]& \big(L^2_r(G),\, \sharp,\, \overline{ \phantom{a} }\big)
    \end{tikzcd}
    \caption{Commutative diagram showing the twisted multiplication and twisted convolution which correspond to operator compositions.}
    \label{fig:H_cd}
\end{figure}
\noindent
Note that for the intermediate space we need to choose
\begin{align*}
    \FW(\mathcal{S}^2) = \FKO^{-1}(L^2_r(G)) \subset L^2_r(G)
\end{align*}
since this is the space that both $\FW$ and $\FKO^{-1}$ map to, not the full $L^2_r(G)$ space.
Before showing that the mappings really respect involution and composition, we show that the intermediate object actually is an $H^*$-algebra.
\begin{lemma}\label{lemma:intermediate_is_Hstar}
    The triple $\big(\FW(\mathcal{S}^2),\, \natural,\,\sqrt{\Delta(\cdot)}\, \overline{ \check{\phantom{a}} }\big)$ is an $H^*$-algebra.
\end{lemma}
\begin{proof}
    Linearity is clear while associativity of $\natural$ can be derived from Lemma \ref{lemma:H*_HS} or \ref{lemma:H*_L2} via \eqref{eq:twisted_conv_from_comp} or \eqref{eq:twisted_conv_from_mul}. That $\sqrt{\Delta(\cdot)}\, \overline{ \check{\phantom{a}}}$ is an involution can be verified manually as
    \begin{align*}
        (f^*)^*(x) = \sqrt{\Delta(x)}\overline{\widecheck{\sqrt{\Delta(\cdot)} \overline{f(\cdot^{-1})}}} = \sqrt{\Delta(x)}\sqrt{\Delta(x^{-1})}\overline{\overline{ f((x^{-1})^{-1}) }} = f(x).
    \end{align*}
    Meanwhile for property \ref{item:involution_of_product}, we can use Lemma \ref{lemma:kirillov_complex_conjugate_final} to verify that, with $f^* = \sqrt{\Delta(\cdot)}\, \overline{ \check{f} }$,
    \begin{align*}
        f^* \natural g^* &= \FKO^{-1} \big( \FKO(f^*) \sharp \FKO(g^*) \big)\\
        &= \FKO^{-1}\big( \overline{\FKO(f)} \sharp \overline{\FKO(g)} \big)\\
        &= \FKO^{-1}\big( \overline{\FKO(g) \sharp \FKO(f)} \big)\\
        &= \sqrt{\Delta(\cdot)} \overline{\widecheck{ \FKO^{-1}( \FKO(g) \sharp \FKO(f) ) }}\\
        &= (g \natural f)^*
    \end{align*}
    as desired. That the involution is an isometry follows from the standard properties of the Haar measure summarized in Lemma \ref{lemma:haar_measure_properties} while submultiplicativity of the twisted convolution can be verified by e.g. pulling back to $\mathcal{S}^2$ using \eqref{eq:twisted_conv_from_comp}. The same holds true for property \ref{item:H_adjoint} by the unitarity of $\FW$.
\end{proof}
All that remains now is to show that each of $\FW$ and $\FKO$ are isometric $\ast$-isomorphisms.
\begin{proposition}
    The mapping $\FW : \big(\mathcal{S}^2,\, \circ,\,{}^*\big) \to \big( \FW(\mathcal{S}^2),\, \natural,\,\sqrt{\Delta(\cdot)}\, \overline{ \check{\phantom{a}} }\big)$ is an isometric $*$-isomorphism between $H^*$ algebras.
\end{proposition}
\begin{proof}
    The mapping being an isometry has already been shown in Proposition \ref{prop:FW_isometry} and the $*$-isomorphism property of multiplication was verified in \eqref{eq:twisted_conv_from_comp}. That the mapping respects the respective involutions can be verified by tracing out the diagram to yield
    \begin{align*}
        \FW(S^*) = \FKO^{-1}(\overline{a_S})
    \end{align*}
    from which we get the desired conclusion through Lemma \ref{lemma:kirillov_complex_conjugate_final}.
\end{proof}

\begin{proposition}
    The mapping $\FKO : \big(\FW(\mathcal{S}^2) ,\, \natural,\,\sqrt{\Delta(\cdot)}\, \overline{ \check{\phantom{a}} }\big) \to \big(L^2_r(G),\, \sharp,\, \overline{ \phantom{a} }\big)$ is an isometric $*$-isomorphism between $H^*$ algebras.
\end{proposition}
\begin{proof}
    The mapping being an isometry follows from Proposition \ref{prop:FKO_unitary} and the $*$-isometry properties follow from Lemma \ref{lemma:kirillov_complex_conjugate_final} and \eqref{eq:twisted_conv_from_mul}.
\end{proof}
Combining these two propositions and Lemma \ref{lemma:intermediate_is_Hstar} with the results of Section \ref{sec:main_iso} allows us to conclude that the diagram in Figure \ref{fig:H_cd} indeed does commute.

\subsection{Wigner distributions}\label{sec:wigner}
As we saw in Section \ref{sec:prelim_weyl}, defining a Wigner distribution and setting up a quantization scheme are really two sides of the same coin. In this article, we take the view that it is the quantization that induces the Wigner distribution, meaning that we will define
\begin{align*}
    W(\psi, \phi)(x) := a_{\psi \otimes \phi}(x) = \FKO(\FW(\psi \otimes \phi))(x)
\end{align*}
for $\psi, \phi \in \mathcal{H}$ and $x \in G$. As we will see shortly, this definition is equivalent to $\langle f, W(\phi, \psi) \rangle_{L^2_r} = \langle A_f \psi, \phi \rangle_\mathcal{H}$ and so should be expected. Before moving on to properties of this function, we expand this definition and the associated inverse to get as close as possible to a computable formula. Using the result of Proposition \ref{prop:fourier_wigner_rank_one}, the fact that $\FW$ can be extended to all of $\mathcal{S}^2$ using Proposition \ref{prop:FW_isometry} and the definition of the Fourier-Kirillov transform \eqref{eq:FKO_def}, we get
\begin{align*}
    W(\psi, \phi)(x) &= \FKO(\FW(\psi \otimes \phi))(x)\\
    &= \frac{1}{\sqrt{|\mathrm{Pf}_F| \cdot \Delta(x)}} \int_{G} \mathcal{W}_{\D \phi} \psi(y) e^{2\pi i \langle K(x^{-1}) F, \log(y) \rangle} \frac{1}{\sqrt{\Theta(\log(y))}}\dmr(y).
\end{align*}
The inverse operation, going from the Wigner distribution to the rank-one operator $\psi \otimes \phi$, can also be written out explicitly as
\begin{align}\nonumber
    \psi \otimes \phi &= \FW^{-1} (\FKO^{-1}(W(\psi, \phi)))\\\label{eq:wigner_inversion}
    &= \FW^{-1}\left( \sqrt{\Theta(\log(\cdot))}\int_{G} W(\psi, \phi)(x) e^{-2\pi i \langle K(x^{-1}) F, \log(\cdot)\rangle}\frac{1}{ \sqrt{|\mathrm{Pf}_F| \cdot \Delta(x)}}\dmr(x) \right)\\\nonumber
    &= \int_{G}\int_{G}  \sqrt{\Theta(\log(y^{-1}))} W(\psi, \phi)(x) e^{-2\pi i \langle K(x^{-1}) F, \log(y^{-1})\rangle} \frac{1}{ \sqrt{|\mathrm{Pf}_F |\cdot \Delta(x)}}\pi(y)\D\dmr(x)\dmr(y).
\end{align}

Next we establish some of the basic properties of this function. 
\begin{proposition}\label{prop:wigner_sesquilinear}
    The mapping $(\psi, \phi) \mapsto W(\psi, \phi)$ is linear in the first argument and anti-linear in the second argument.
\end{proposition}
\begin{proof}
    Using that $W(\psi, \phi) = a_{\psi \otimes \phi}$, we can see that this holds because $(\psi, \phi) \mapsto \psi \otimes \phi$ is linear in the first argument and antilinear in the second argument while $a : \mathcal{S}^2 \to L^2_r(G)$ is linear by Theorem \ref{theorem:quantization_main}.
\end{proof}

\begin{proposition}\label{prop:wigner_props}
    Let $\psi, \phi \in \mathcal{H}$, then the Wigner distribution $W(\psi, \phi)$ belongs to $L^2_r(G)$ and satisfies the weak relation
    \begin{align}\label{eq:wigner_weak_quant}
        \big\langle f, W(\phi, \psi) \big\rangle_{L^2_r} = \langle A_f \psi, \phi \rangle_{\mathcal{H}}
    \end{align}
    for $f \in L^2_r(G)$, as well as the identities
    \begin{align*}
        \overline{W(\psi, \phi)(x)} &= W(\phi, \psi)(x),\\
        R_x W(\psi, \phi) &= W(\pi(x)\psi, \pi(x)\phi).
    \end{align*}
\end{proposition}
\begin{proof}
    We apply quantization to both sides of \eqref{eq:wigner_weak_quant} and use the definition of inner products in the space of Hilbert-Schmidt operators and the definition of rank-one operators to get
    \begin{align*}
        \langle f, W(\phi, \psi)\rangle_{L^2_r} &= \langle A_f, \phi \otimes \psi\rangle_{\mathcal{S}^2}\\
        &= \sum_n \langle A_f (\psi \otimes \phi)\xi_n, \xi_n \rangle_\mathcal{H}\\
        &= \sum_n \langle \xi_n, \phi \rangle \langle A_f \psi, \xi_n \rangle_\mathcal{H}\\
        &= \langle A_f \psi, \phi \rangle_\mathcal{H}.
    \end{align*}
    For the first of the two identities, we let $f \in L^2_r(G)$ be arbitrary and look at the weak action of $W(\psi, \phi)$ to conclude
    \begin{align*}
        \langle f, W(\psi, \phi) \rangle_{L^2_r} &= \langle A_f \phi, \psi \rangle_\mathcal{H}\\
        &= \overline{\langle A_{\bar{f}} \psi, \phi \rangle_\mathcal{H}}\\
        &= \big\langle f, \overline{W(\psi, \phi)} \big\rangle_{L^2_r}.
    \end{align*}
    Since $f$ was arbitrary, it follows that $W(\psi, \phi) = \overline{W(\phi, \psi)}$.

    Lastly, we compute the quantization of $R_x W(\psi, \phi)$ using Proposition \ref{prop:quantization_translation} as
    \begin{align*}
        R_x W(\psi, \phi) = R_x a_{\psi \otimes \phi} = a_{\pi(x^{-1})^* (\psi \otimes \phi) \pi(x^{-1})} = a_{\pi(x) \psi \otimes \pi(x)\phi} = W(\pi(x)\psi, \pi(x) \phi)
    \end{align*}
    which completes the proof.
\end{proof}
The weak expression for the Wigner distribution \eqref{eq:wigner_weak_quant} shows that our quantization procedure does not preserve positivity for the same reason as in the Weyl-Heisenberg case: For a non-negative $f$, $\langle A_f \psi, \psi \rangle = \langle f, W(\psi) \rangle$ can be negative at a point if $W(\psi)$ is negative which is the case for most $\psi$ in the Weyl-Heisenberg case \cite{grochenig2013foundations, Grchenig2019}.

By letting $f$ be another Wigner distribution in \eqref{eq:wigner_weak_quant} we immediately get the following generalization of Theorem \ref{theorem:moyal_wigner}.
\begin{proposition}\label{prop:wigner_orthog}
    Let $\psi_1, \phi_1, \psi_2, \phi_2 \in \mathcal{H}$, then
    \begin{align*}
        \big\langle W(\psi_1, \phi_1), W(\psi_2, \phi_2) \big\rangle_{L^2_r} = \langle \psi_1, \psi_2\rangle_\mathcal{H} \overline{\langle \phi_1, \phi_2 \rangle_\mathcal{H}}.
    \end{align*}
\end{proposition}
From the above proposition it is easy to see that for an orthonormal sequence $(\phi_n)_n$, the sequence $\big(W(\phi_n, \phi_m)\big)_{n,m}$ is also orthonormal. In fact, the property of being complete is also preserved through this mapping which generalizes the results \cite[Prop.~188]{deGosson2011} and \cite[Lem.~7.4]{Berge2022affine}.
\begin{corollary}\label{cor:orth_basis_from_onb}
    Let $(\phi_n)_n$ be an orthonormal basis for $\mathcal{H}$, then the sequence $(W(\phi_n, \phi_m))_{n,m}$ is an orthonormal basis for $L^2_r(G)$.
\end{corollary}
\begin{proof}
    As remarked above, preserving orthonormality follows directly from Proposition \ref{prop:wigner_orthog}. For completeness, suppose $f \in L^2_r(G)$ is orthogonal to $W(\phi_n, \phi_m)$ for all $n$ and $m$. Then by \eqref{eq:wigner_weak_quant}, we also have that
    \begin{align*}
        0 = \langle f, W(\phi_n, \phi_m) \rangle_{L^2_r} = \langle A_f \phi_m, \phi_n \rangle_{\mathcal{H}}
    \end{align*}
    for all $n,m$. However, since $(\phi_n)_n$ is an orthonormal basis for $\mathcal{H}$, this means that $A_f$ must be the zero operator which in turn implies that $f$ is the zero function, finishing the proof.
\end{proof}
From Proposition \ref{prop:wigner_orthog} we also get a proof that the Wigner distribution is uniquely determined by the window $\psi$ up to a unimodular constant. This is well known in the Weyl-Heisenberg case, see e.g. \cite[Prop.~1.98]{folland1989harmonic}. Our proof will follow that for the affine Wigner distribution in \cite[Cor.~3.3]{Berge2022affine}.
\begin{proposition}\label{prop:unique_wigner}
    Let $\psi, \phi \in \mathcal{H}$. Then $W(\psi) = W(\phi)$ if and only if $\psi = c \cdot \phi$ with $|c| = 1$.
\end{proposition}
\begin{proof}
    If $\psi = c \cdot \phi$, the result is clear as $c\phi \otimes c \phi = \phi \otimes \phi$ if $|c| = 1$. For the other direction, note that 
    \begin{align*}
        \big|\langle \psi, \phi \rangle_\mathcal{H}\big|^2 = \big\langle W(\psi), W(\phi) \big\rangle_{L^2_r} = \underbrace{\Vert W(\psi) \Vert_{L^2_r}^2}_{=\Vert \psi \Vert_{\mathcal{H}}^4} = \underbrace{\Vert W(\phi) \Vert_{L^2_r}^2}_{= \Vert \phi \Vert_{\mathcal{H}}^4} = \Vert \psi \Vert_\mathcal{H}^2 \Vert \phi \Vert_\mathcal{H}^2.
    \end{align*}
    Consequently, if we apply Cauchy-Schwarz to the left-hand side, we must have equality which implies that $\psi$ and $\phi$ are collinear.
\end{proof}
We can also deduce the following fact about the sums of Wigner distributions.
\begin{proposition}\label{prop:weak_I}
    Let $(\phi_k)_k$ be an orthonormal basis of $\mathcal{H}$, then $A_{\sum_k^n W(\phi_k)}$ converges weakly to $I_\mathcal{H}$ as $n \to \infty$.
\end{proposition}
\begin{proof}
    This ends up being a straight-forward verification. Indeed, for any $\psi, \phi \in \mathcal{H}$,
    \begin{align*}
        \lim_{n \to \infty} \big\langle A_{\sum_k^n W(\phi_k)} \psi, \phi \big\rangle &= \lim_{n\to\infty} \sum_{k=1}^n \big\langle W(\phi_k), W(\phi, \psi) \big\rangle\\
        &= \sum_{k=1}^\infty \langle \phi_k, \phi \rangle \overline{\langle \phi_k, \psi \rangle} = \langle \psi, \phi \rangle
    \end{align*}
    where we used the linearity of quantization, Proposition \ref{prop:wigner_props} and Proposition \ref{prop:wigner_orthog}.
\end{proof}
Lastly we return to the discussion on positivity and note that the following characterization of positivity is immediate from Proposition \ref{prop:wigner_props}.
\begin{proposition}
    An operator $S \in \mathcal{S}^2$ is positive if and only if
    \begin{align*}
        \int_G a_S(x) W(\psi)(x) \dmr(x) \geq 0\qquad \text{ for all }\psi \in \mathcal{H}.
    \end{align*}
\end{proposition}
Note that in the Weyl-Heisenberg case, Wigner distributions $W(\psi)$ are not positive everywhere in general \cite{hudson1974wigner} so positivity of an operator is not equivalent to $a_S \geq 0$.

%\begin{proposition}
%    Let $G=R\rtimes_\phi H$ and $\psi_1,\psi_2\in \mathcal{H}$. Then the Wigner distribution is given by\todo{Add this!}
%\end{proposition}
%\begin{proof}
%Using Proposition \ref{semi:simple} and Proposition \ref{semi:simple2} we get that
%    \begin{align*}
%        W(\psi_1,\psi_2)(r_1,h_2)&=\FKO(\FW(\psi_1\otimes\psi_2))\\
%       &= \int_R  \psi_1(s)\overline{\psi_2(tr^{-1})}\frac{1}{\sqrt{\Delta(tr^{-1})}}e^{2\pi i F\log_H(\phi(tr^{-1})h)}d\mu_r^R(t)
%    \end{align*}
%\end{proof}

\subsection{Operator convolutions from quantum harmonic analysis}\label{sec:operator_convolution_quantization}
Our interest in the quantization map partly stems from its interaction with the operator convolutions from quantum harmonic analysis, discussed in Section \ref{sec:prelim_qha}. In this section we will show that the convolution-quantization relations outlined in Section \ref{sec:prelim_qha_weyl_relation} can be generalized beyond the Weyl-Heisenberg group. Corresponding results were earlier shown for affine quantization in  \cite[Sec.~3.2]{Berge2022}.

First we show that quantization respects function-operator convolutions in a nice manner.
\begin{proposition}\label{prop:conv_quantization}
    Let $f \in L^1_r(G)$ and $g \in L^2_r(G)$. Then
    $$
    A_{f * g} = f \star A_g.
    $$
\end{proposition}
\begin{proof}
We explicitly compute the dequantization $a_{f \star A_g}$ and show that it is equal to $f * g$ which suffices by injectivity. To move the dequantization inside the integral defining $f \star A_g$, we use that bounded operators commute with convergent Bochner integrals (see \cite{folland2016course} or \cite[Prop.~2.4]{Luef2018}) which yields
\begin{align*}
    a_{f \star A_g}(x) &= \FKO\left( \FW\left( \int_G f(y) \pi(y)^* A_g \pi(y)\dmr(y) \right) \right)(x)\\
    &= \int_G f(y) \FKO\big( \FW ( \pi(y)^* A_g \pi(y))\big)(x) \dmr(y)\\
    &= \int_G f(y) R_{x^{-1}} g(y) \dmr(y)\\
    &= \int_G f(y) g(x y^{-1})\dmr(y) = f * g
\end{align*}
where we used Proposition \ref{prop:quantization_translation} for the translation.
\end{proof}
For operator-operator convolutions we will need to keep track of an involution on one of the functions in the same way as for affine quantization \cite[Prop.~3.7]{Berge2022}.
\begin{proposition}\label{prop:op_op_quantization}
    Let $f,g \in L^2_r(G)$, then
    $$
    A_f \star A_g = f * \check{g}.
    $$
\end{proposition}
\begin{proof}
    We compute
    \begin{align*}
        (A_f \star A_g)(x) &= \tr\big(A_f \pi(x)^* A_g \pi(x)\big)\\
        &= \big\langle A_f, \pi(x)^* A_{\bar{g}} \pi(x) \big\rangle_{\mathcal{S}^2}\\
        &= \big\langle f, R_{x^{-1}}\bar{g} \big\rangle_{L^2_r}\\
        &= \int_G f(y) \check{g}(xy^{-1})\dmr(y) = f * \check{g}(x)
    \end{align*}
    where we used that $A_f^* = A_{\bar{f}}$ from Proposition \ref{prop:quantization_adjoint}.
\end{proof}
The rank-one realization of this result is of independent interest and implies a relation for the Wigner distribution which is well known in the Weyl-Heisenberg case.
\begin{corollary}\label{cor:wigner_convolution_wavelet}
    Let $\psi_1, \psi_2, \phi_1, \phi_2 \in \mathcal{H}$, then
    \begin{align*}
        W(\psi_1, \psi_2) * \widecheck{W(\phi_2, \phi_1)}(x) = \mathcal{W}_{\phi_1} \psi_1(x) \overline{\mathcal{W}_{\phi_2} \psi_2(x)}.     
    \end{align*}
    In particular,
    \begin{align*}
        W(\psi) * \widecheck{W(\phi)}(x) = |\mathcal{W}_{\phi} \psi(x)|^2.
    \end{align*}
\end{corollary}
\begin{proof}
    Since $A_{W(\psi_1, \psi_2)} = \psi_1 \otimes \psi_2$, we can use Proposition \ref{prop:op_op_quantization} to write the convolution as
    \begin{align*}
        W(\psi_1, \psi_2) * \widecheck{W(\phi_2, \phi_1)}(x) &= (\psi_1 \otimes \psi_2) \star (\phi_2 \otimes \phi_1)(x)\\
        &= \tr\big((\psi_1 \otimes \psi_2) \pi(x)^* (\phi_2 \otimes \phi_1) \pi(x)\big)\\
        &= \sum_n \langle \pi(x)^*(\phi_2 \otimes \phi_1) \pi(x) e_n, \psi_2 \rangle \langle \psi_1, e_n \rangle\\
        &= \sum_n \langle e_n, \pi(x)^*\phi_1 \rangle \langle \phi_2, \pi(x)\psi_2 \rangle \langle \psi_1, e_n  \rangle\\
        &= \langle \psi_1, \pi(x)^* \phi_1 \rangle \overline{\langle \psi_2, \pi(x)^* \phi_2 \rangle} = \mathcal{W}_{\phi_1}\psi_1(x) \overline{\mathcal{W}_{\phi_2} \psi_2(x)}
    \end{align*}
    which is what we wished to show.
\end{proof}

By combining Proposition \ref{prop:op_op_quantization} and Proposition \ref{prop:conv_quantization} we can also get an expression for $A_{T \star S}$. However, to generalize the classical result for the Weyl-Heisenberg group we need a notion of parity conjugating an operator. Since we don't have a parity operator in the general case, we cannot approach this issue head-on but can instead define $\check{S}$ using the quantization $a_S$ of $S$ so that it generalizes the classical property $A_{\check{a_S}} = \check{S}$ \cite[Lem.~3.2 (ii)]{Luef2018}. An immediate consequence of this definition is that $a_{\check{S}} = \check{a_S}$. We also get the following corollary from the earlier results. 
\begin{corollary}
    Let $T, S \in \mathcal{S}^2$ be such that $a_T \in L^1_r(G)$, then
    \begin{align*}
        A_{T \star S} = a_T \star A_{\check{a_S}} = a_T \star \check{S}.
    \end{align*}
\end{corollary}
\begin{proof}
    By first using that $T = A_{a_T}$, $S = A_{a_S}$ and then applying Proposition \ref{prop:op_op_quantization} followed by Proposition \ref{prop:conv_quantization}, we get
    \begin{align*}
        A_{T \star S} = A_{A_{a_T} \star A_{a_S}} = A_{a_T * \check{a_S}} = a_T \star A_{\check{a_S}} = a_T \star \check{S}
    \end{align*}
    since this is how we defined $\check{S}$.
\end{proof}

\subsection{Unimodular trace formula}
In 2000, Du and Wong \cite{Du2000} were the first to show that for Weyl quantization on the Weyl-Heisenberg group, $\tr(A_f) = \int_{\R^{2d}} f(x)\,\mathrm{d}x$ which is called a \emph{trace formula}, see \cite[Prop.~286]{deGosson2011} and \cite[Cor.~6.2.1]{Luef2018} for other proofs. This result was generalized to the affine group in \cite[Prop.~4.14]{Berge2022} in two ways,
\begin{align}\label{eq:affine_trace_formula}
    \int_G f(x)\dmr(x) = \tr(A_f),\qquad \int_G f(x)\dml(x) = \tr(\D^{-1} A_f \D^{-1}).   
\end{align}
We will show the corresponding result under the assumption that the group is unimodular but remark that \eqref{eq:affine_trace_formula} is likely to hold in the general case as well.

\begin{theorem}\label{theorem:budget_trace_integral_relation}
    Assume that the underlying group is unimodular and let $f \in L^1(G)$ be such that $A_f \in \mathcal{S}^1$, then
    \begin{align}\label{eq:unimod_int_eq_tr}
        \frac{1}{\sqrt{\mathrm{Pf}_F}}\int_G f(x)\dm(x) = \tr(A_f).
    \end{align}
\end{theorem}
\begin{proof}
    We will begin by showing the first relation in the special case where $f$ is continuous with compact support. Looking at the definition of the inverse Fourier-Kirillov transform \eqref{eq:inverse_FKO}, we see that the integral can then be identified with $\FKO^{-1}(f)(e)$. That it is okay to evaluate $\FKO^{-1}(f)$ at a point follows from the Riemann-Lebesgue lemma and that the mappings in Proposition \ref{prop:FKO_unitary} are well behaved when the group is unimodular. Since $f \in L^2(G)$ we can also identify $f$ with $a_{A_f} = \FKO(\FW(A_f))$. Now by applying Lemma \ref{lemma:invert_FKO_and_FW}, we get
    \begin{align*}
        \frac{1}{\sqrt{\mathrm{Pf}_F}}\int_G f(x)\dm(x) &= \FKO^{-1}(\FKO(\FW(A_f)))(e)\\
        &=\FW(\FW^{-1}(\FW(A_f)))(e)\\
        &=\FW(A_f)(e)\\
        &= \tr(A_f)
    \end{align*}
    where we used the continuity of $\FW(A_f)$ afforded by Proposition \ref{prop:fw_riemann_lebesgue} in the last step.

    To drop the assumption of continuity and compact support of $f$, fix a continuous function $g$ with compact support. From \eqref{eq:qha_op_op_conv} in Section \ref{sec:prelim_qha} and Proposition \ref{prop:op_op_quantization} above, we then get that
    \begin{align*}
        \tr(A_f) \tr(A_g) &= \int_G A_f \star A_g(x)\dm(x)\\
        &=\int_G f * \check{g} (x) \dm(x)\\
        &=\int_G f(x) \dm(x) \int_G \check{g}(x) \dm(x) = \int_G f(x) \dm(x) \int_G g(x) \dm(x)
    \end{align*}
    where we in the last step used the unimodularity of the group. Now dividing by $\tr(A_g)$ on both sides and using that $\tr(A_g) = \frac{1}{\sqrt{\mathrm{Pf}_F}}\int_G g(x)\dm(x)$, we get the desired equality.
\end{proof}
\begin{remark}
    One immediate consequence of this result is that for a Wigner distribution $W(\psi, \phi)$,
    \begin{align*}
        \frac{1}{\sqrt{\mathrm{Pf}_F}}\int_G W(\psi, \phi)(x)\dm(x) = \langle \psi, \phi \rangle.
    \end{align*}
\end{remark}
With the above result, we can obtain a briefer (albeit less elementary) proof of \cite[Thm.~4.4.6 (a)]{grochenig2013foundations} and generalize it to all unimodular groups.
\begin{corollary}
    Assume that the underlying group is unimodular and suppose that $f \in L^2(G) \cap L^1(G)$ is such that $A_f \in \mathcal{S}^1$ and $A_f$ is a positive operator, then
    $$
    \Vert f \Vert_{L^2(G)} \leq \frac{1}{\sqrt{\mathrm{Pf}_F}}\int_{G} f(x)\dm(x).
    $$
\end{corollary}
\begin{proof}
    Since $A_f \in \mathcal{S}^1$ is a positive operator, its singular value decomposition is of the form $A_f = \sum_{n} \lambda_n (\phi_n \otimes \phi_n)$ where $(\phi_n)_n$ is some orthonormal basis of $\mathcal{H}$ and all $\lambda_n$ are non-negative. By Theorem \ref{theorem:budget_trace_integral_relation}, we can write
    \begin{align*}
        \frac{1}{\sqrt{\mathrm{Pf}_F}}\int_G f(x)\dm(x) = \tr(A_f) = \sum_n \lambda_n = \big\Vert (\lambda_n)_n \big\Vert_{\ell^1} \geq \big\Vert (\lambda_n)_n \big\Vert_{\ell^2} = \Vert A_f \Vert_{\mathcal{S}^2} = \Vert f \Vert_{L^2(G)}
    \end{align*}
    where we used the embedding $\ell^1 \subseteq \ell^2$ and that quantization is an isometry.
\end{proof}
This result implies another immediate corollary.
\begin{corollary}
    Suppose that $0 \neq f \in L^2(G) \cap L^1(G)$ is non-positive everywhere and such that $A_f \in \mathcal{S}^1$, then $A_f$ cannot be a positive operator.
\end{corollary}

\section{Applications}\label{sec:applications}
Classical Weyl quantization has historically proved to be a valuable tool to tackle a wide variety of problems. In this section, we set out to show how Weyl quantization on exponential groups can be used to treat analog problems. Specifically we will show three results which all have counterparts with the classical Weyl transform and the Weyl-Heisenberg group but for which the generalizations are novel and one proof which is simplified by the tools developed in this article.

\subsection{Wavelet phase retrieval}\label{sec:phase_retrieval}
The problem of \emph{phase retrieval} in the context of time-frequency analysis may be summarized as inverting the mapping
\begin{align}\label{eq:phase_ret_ex}
    |\mathcal{W}_\phi|^2 : \psi \mapsto |\langle \psi, \pi(\cdot)^*\phi\rangle|^2.
\end{align}
Since $\psi$ can be recovered from $\mathcal{W}_\phi \psi$, this task is equivalent to recovering the phase from the (complex-valued) wavelet transform $\mathcal{W}_\phi \psi$, hence the name. Historically, this problem has been mainly considered for the Weyl-Heisenberg group where we are asked to recover the phase of the short-time Fourier transform $V_\phi \psi$. Outside of the Weyl-Heisenberg case, characterizing solvability of the problem has been notoriously difficult and has been an open question for many years. In the case of the affine group where $|\mathcal{W}_\phi \psi |^2$ is called the \emph{scalogram}, there have been partial results involving sufficient conditions on the wavelet $\phi$ but there is no general condition as in the spectrogram case. This is an active area of research with many applications and recent results, see e.g. \cite{Waldspurger2017, Mallat2015, Alaifari2023, Fhr2023}.

In the Weyl-Heisenberg case, solvability is characterized by the Fourier transform of $W(\phi)$ being nonvanishing. This can be seen from the formula 
\begin{align}\label{eq:phase_ret_wigner_stft}
    |V_\phi \psi|^2 = W(\psi) * \widecheck{W(\phi)}
\end{align}
from Corollary \ref{cor:wigner_convolution_wavelet}, by taking the Fourier transform on both sides and dividing away the Fourier transform of $\widecheck{W(\phi)}$. In the general case we do not have such a nice Fourier convolution theorem to rely on but we still have the relation
\begin{align}\label{eq:phase_ret_wigner_expression}
    |\mathcal{W}_\phi \psi(x)|^2 = (\psi \otimes \psi) \star (\phi \otimes \phi)(x) = W(\psi) * \widecheck{W(\phi)}(x).
\end{align}

As we saw in Proposition \ref{prop:unique_wigner}, the Wigner distribution $W(\psi)$ determines $\psi$ up to a unimodular constant. Formally, we can recover $\psi$ up to this phase constant by quantizing $W(\psi)$ to get $\psi \otimes \psi$ and applying the resulting operator to any element $\xi$ which is not orthogonal to $\psi$ as outlined in \eqref{eq:wigner_inversion}. This is written as
\begin{align}\label{eq:wigner_to_function}
    \psi = \frac{A_{W(\psi)} \xi}{ \Vert A_{W(\psi)} \xi \Vert}
\end{align}
since $A_{W(\psi)} = \psi \otimes \psi$. Consequently, we wish to find conditions on $\phi$ so that we may manipulate \eqref{eq:phase_ret_wigner_expression} as to get an explicit expression for $W(\psi)$. This requires being able to ``disentangle'' the squared wavelet transform $|\mathcal{W}_\phi \psi|^2$, our tool for this will be the following lemma.
\begin{lemma}
    Let $f, g \in L^1_l(G)$ be such that $f \Delta \in L^1_r(G)$ and $\pi, \pi_r$ be the left and right integrated representations. Then
    \begin{align*}
        \pi(f * g) = \pi(g) \pi_r(f\Delta).
    \end{align*}
\end{lemma}
\begin{remark}
    Recall that $f*g$ here and throughout the paper denotes \emph{right} convolution.    
\end{remark}
\begin{proof}
    Using the substitution $z = xy^{-1} \implies x = zy$, we compute
    \begin{align*}
        \pi(f * g) &= \int_G (f*g)(x) \pi(x)\dml(x) = \int_G \left( \int_G f(y) g(xy^{-1})\dmr(y) \right)\pi(x)\dml(x)\\
        &=\int_G \int_G f(y) g(z) \pi(z) \pi(y)\dmr(y)\dml(zy)\\
        &=\left( \int_G g(z) \pi(z)\dml(z) \right)\left( \int_G f(y) \Delta(y) \pi(y)\dmr(y) \right) = \pi(g) \pi_r(f \Delta).
    \end{align*}
\end{proof}
We will also need the following auxiliary lemma which is easily verified.
\begin{lemma}
    Let $f\in L^1_r(G)$ and $g \in L^1_l(G)$, then
    \begin{align*}
        \widecheck{f * \check{g}}(x) = g * \check{f}(x).
    \end{align*}
\end{lemma}
\begin{proof}
    With the substitution $z = yx$, we compute
    \begin{align*}
        \widecheck{f * \check{g}}(x) = f * \check{g}(x^{-1}) &= \int_G f(y) \check{g}(x^{-1} y^{-1})\dmr(y)\\
        &= \int_G f(y) g(yx)\dmr(y)\\
        &= \int_G g(z) f(zx^{-1})\dmr(zx^{-1})\\
        &= \int_G g(z) \check{f}(xz^{-1})\dmr(z) = g * \check{f}(x)
    \end{align*}
    which is what we wished to show.
\end{proof}
With these two lemmas, we will be able to wrangle \eqref{eq:phase_ret_wigner_expression} into a form where $\psi$ can conceivably be recovered.
\begin{theorem}\label{theorem:phase_retrieval}
    If $\phi \in \mathcal{H}$ is such that $W(\phi) \in L^1_l(G)$, $W(\phi) \Delta \in L^1_r(G)$ and the right integrated representation $\pi_r(W(\phi) \Delta)$ has a right inverse, then the mapping $\psi \mapsto |\mathcal{W}_\phi \psi|^2$ for $\psi$ such that $W(\psi) \in L^1_r(G)$ can be inverted up to a global phase as
    \begin{align*}
        W(\psi) = \FW \Big( \pi\big(\widecheck{|\mathcal{W}_\phi \psi|^2}\big) \pi_r\big( W(\phi) \Delta \big)^{-1} \D \Big)
    \end{align*}
    followed by
    \begin{align}\label{eq:phase_retrieval_explicit}
        \psi = \frac{\FW^{-1}\big(\FKO^{-1}\big(W(\psi)\big)\big) \xi}{ \Vert \FW^{-1}\big(\FKO^{-1}\big(W(\psi)\big)\big) \xi \Vert}
    \end{align}
    where $\xi \in \mathcal{H}$ is any vector which is not orthogonal to $\psi$.
\end{theorem}
\begin{proof}
    We compute the left integrated representation of the squared modulus of the wavelet transform as
    \begin{align*}
        \pi\big(\widecheck{|\mathcal{W}_\phi \psi|^2}\big) &= \pi\big( \widecheck{W(\psi) * \widecheck{W(\phi)}} \big)\\
        &= \pi\big( W(\phi) * \widecheck{W(\psi)} \big)\\
        &= \pi\big( \widecheck{W(\psi)} \big) \pi_r\big( W(\phi) \Delta \big).
    \end{align*}
    Now if $\pi_r\big( W(\phi) \Delta \big)$ has a right inverse, we can apply it to $\pi\big(\widecheck{|\mathcal{W}_\phi \psi|^2}\big)$ to get that
    \begin{align*}
        \pi\big(\widecheck{|\mathcal{W}_\phi \psi|^2}\big) \pi_r\big( W(\phi) \Delta \big)^{-1} &= \pi\big( \widecheck{W(\psi)} \big)\\
        \implies \FW^{-1}\big(W(\psi)\big) &= \pi\big(\widecheck{|\mathcal{W}_\phi \psi|^2}\big) \pi_r\big( W(\phi) \Delta \big)^{-1} \D\\
        \implies W(\psi) &= \FW \Big( \pi\big(\widecheck{|\mathcal{W}_\phi \psi|^2}\big) \pi_r\big( W(\phi) \Delta \big)^{-1} \D \Big).
    \end{align*}
    From here the explicit expression \eqref{eq:phase_retrieval_explicit} follows from applying \eqref{eq:wigner_to_function}.
\end{proof}
\begin{remark}
    The same argument can be used to show invertibility of the mapping $\psi \mapsto Q_S(\psi)$ where $Q_S$ is the Cohen's class distribution with operator window $S$ (see \cite[Sec.~5.1]{Halvdansson2023} and \cite[Sec.~7.5]{Luef2019}). In that case, the condition becomes right invertibility of $\pi_r(a_S \Delta)$.
\end{remark}

\subsection{Wigner approximation problem}\label{sec:wigner_approximation}
The problem of approximating arbitrary $L^2(\R^{2d})$ functions by Wigner distributions was first considered for the Weyl-Heisenberg group in \cite{ben2018wigner} and later extended to the affine setting in \cite[Sec.~8.1]{Berge2022affine}. Since we have constructed a Wigner distribution in our setting, we can formulate the corresponding problem as follows: Given a group $G$ with associated Hilbert space $\mathcal{H}$, how close is a function $f \in L^2_r(G)$ to the \emph{Wigner space}
\begin{align*}
    \mathfrak{W}(G) = \big\{ W(\psi) : \psi \in \mathcal{H} \big\} \subset L^2_r(G)?
\end{align*}
As self-adjoint rank-one operators are a proper subset of the space of Hilbert-Schmidt operators, the set $\mathfrak{W}(G)$ is also a proper subset of $L^2_r(G)$. Moreover, using the orthogonality relations for Wigner distributions, Proposition \ref{prop:wigner_orthog}, we can see that this space is closed. 

In \cite{Berge2022affine}, the role of quantization in solving the problem is emphasized and their approach carries over with minor modifications to our setting since we have similar quantization properties. The main idea can be summarized as ``\textit{via quantization, being close to being a Wigner distribution corresponds to almost being a rank-one operator}''.
\begin{theorem}
    Fix a real-valued $f \in L^2_r(G)$ with quantization $A_f$ and let $\lambda_{\max}^+(A_f)$ be the positive part of $\max_{\lambda \in \operatorname{Spec}(A_f)} \lambda$, then
    \begin{align*}
        \inf_{g \in \mathfrak{W}(G)} \Vert f-g \Vert_{L^2_r} = \sqrt{\Vert f \Vert_{L^2_r}^2 - \lambda_{\max}^+(A_f)^2}.
    \end{align*}
    The infimum is always attained and the number of unique minimizers (up to multiplication by a unimodular constant) is equal to the multiplicity of $\lambda_{\max}^+(A_f)$.
\end{theorem}
In the interest of completeness, we write out a proof of the above theorem but remark that we follow the same strategy as in \cite{Berge2022affine}.
\begin{proof}
    By quantization being a unitary isometry, we have that
    \begin{align}\label{eq:quantization_wigner_approx_prob_reform}
        \inf_{g \in \mathfrak{W}(G)} \Vert f-g \Vert_{L^2_r} = \inf_{\psi \in \mathcal{H}} \Vert A_f - \psi \otimes \psi \Vert_{\mathcal{S}^2}.
    \end{align}
    We wish to apply the spectral theorem to $A_f$ and so first note that it is compact by virtue of being a Hilbert-Schmidt operator and self-adjoint by real-valuedness via Proposition \ref{prop:quantization_adjoint}. Hence its eigendecomposition can be written as
    \begin{align*}
        A_f = \sum_k \lambda_k (\phi_k \otimes \phi_k)
    \end{align*}
    with convergence in the Hilbert-Schmidt norm. We can now rewrite \eqref{eq:quantization_wigner_approx_prob_reform} as
    \begin{align*}
        \inf_{g \in \mathfrak{W}(G)} \Vert f-g \Vert_{L^2_r} = \inf_{\psi \in \mathcal{H}} \left\Vert \sum_k \lambda_k (\phi_k \otimes \phi_k) - \psi \otimes \psi \right\Vert_{\mathcal{S}^2}.
    \end{align*}
    By the orthogonality of $(\phi_k)_k$, this quantity is minimized when $\psi = \sqrt{\lambda_j} \phi_j$ where $\lambda_j$ and $\phi_j$ are the eigendata corresponding to the largest positive eigenvalue. The reformulation as $\sqrt{\Vert f \Vert_{L^2_r}^2 - \lambda_{\max}^+(A_f)^2}$ also follows from the orthogonality of $(\phi_k)_k$. Lastly the statement about non-unique minimizers follows from the ambiguity associated with choosing $\lambda_j$ and $\phi_j$.
\end{proof}

\subsection{Wavelet/Wigner spaces have trivial intersection}\label{sec:wavelet_intersection}
The intersections of wavelet spaces, meaning the images of $\mathcal{H}$ under the mapping $\mathcal{W}_\phi : \psi \mapsto \langle \psi, \pi(\cdot) \phi \rangle$, have been studied in \cite{cont_wavelet_transform, berge2020interpolation} where it was shown that the intersection $\mathcal{W}_{\phi_1}(\mathcal{H}) \cap \mathcal{W}_{\phi_2}(\mathcal{H})$ is trivial unless $\phi_1 = c\phi_2$ for some complex number $c$. Recently, Skrettingland and Luef provided a simplified proof for Gabor spaces in \cite[Lem.~3.3]{luef2021jfa} which depended on a Weyl-Heisenberg version of our Proposition \ref{prop:fourier_wigner_rank_one}. Below we show that Proposition \ref{prop:fourier_wigner_rank_one} can simplify the proof of \cite[Thm.~4.2]{cont_wavelet_transform} in our setting of exponential groups.
\begin{proposition}
    Let $\phi_1, \phi_2 \in \operatorname{Dom}(\D^{-1})$. If there exists $c \in \mathbb{C}$ such that $\phi_1 = c\phi_2$, then $\mathcal{W}_{\phi_1}(\mathcal{H}) = \mathcal{W}_{\phi_2}(\mathcal{H})$. Otherwise, $\mathcal{W}_{\phi_1}(\mathcal{H}) \cap \mathcal{W}_{\phi_2}(\mathcal{H}) = \{ 0 \}$.
\end{proposition}
\begin{proof}
    The forward implication follows from the relation $\mathcal{W}_{c\phi_2}(\xi) = \mathcal{W}_{\phi_2}(\Bar{c}\xi)$. For the other direction, we have the following chain of implications
    \begin{align*}
        \mathcal{W}_{\phi_1}(\psi_1) &= \mathcal{W}_{\phi_2}(\psi_2)\\
        \implies \langle \psi_1, \pi(\cdot)^* \phi_1 \rangle &= \langle \psi_2, \pi(\cdot)^* \phi_2 \rangle && \text{(Definition of } \mathcal{W}_\phi(\psi) \text{)}\\
        \implies \FW(\psi_1 \otimes \D^{-1}\phi_1) &= \FW(\psi_2 \otimes \D^{-1} \phi_2) && \text{(Applying } \FW \text{ to both sides)}\\
        \implies \psi_1 \otimes \D^{-1}\phi_1 &= \psi_2 \otimes \D^{-1} \phi_2 && \text{(Injectivity of } \FW \text{)}\\
        \implies \D^{-1} \phi_1 \otimes \psi_1 &= \D^{-1} \phi_2 \otimes \psi_2 && \text{(Taking adjoints)}\\
        \implies \Vert \psi_1 \Vert^2 \D^{-1} \phi_1 &= \langle \psi_1, \psi_2 \rangle \D^{-1} \phi_2 && \text{(Applying to $\psi_1$)}\\
        \implies \phi_1 &= \frac{\langle \psi_1, \psi_2 \rangle}{\Vert \psi_1 \Vert^2} \phi_2 && \text{(Solving for } \phi_1 \text{)}
    \end{align*}
    which finishes the proof.
\end{proof}

Cross-Wigner distributions $W(\psi, \phi)$ are sometimes used as time-frequency distributions with $\phi$ being viewed as the window. For these purposes, define the cross-Wigner space $W(\mathcal{H}, \phi)$ as 
\begin{align*}
    W(\mathcal{H}, \phi) = \big\{ W(\psi, \phi) : \psi \in \mathcal{H} \big\} \subset L^2_r(G).
\end{align*}
In this case we have the same type of result as for wavelet spaces with an even shorter proof.
\begin{proposition}
    Let $\phi_1, \phi_2 \in \mathcal{H}$. If there exists $c \in \mathbb{C}$ such that $\phi_1 = c\phi_2$, then $W(\mathcal{H}, \phi_1) = W(\mathcal{H}, \phi_2)$. Otherwise $W(\mathcal{H}, \phi_1) \cap  W(\mathcal{H}, \phi_2) = \{ 0 \}$.
\end{proposition}
\begin{proof}
    The forward implication again follows from $W(\xi, c\phi_2) = W(\bar{c}\xi, \phi_2)$ by the sesquilinearity of the cross-Wigner distribution from Proposition \ref{prop:wigner_sesquilinear}.

    Suppose there exist $\psi_1, \psi_2 \in \mathcal{H}$ such that $W(\psi_1, \phi_1) = W(\psi_2, \phi_2)$. We will show that this implies that $\phi_1 = c \phi_2$. Indeed, by the equality
    \begin{align*}
        \big\langle W(\psi_1, \phi_1), W(\psi_2, \phi_2)\big\rangle = \langle \psi_1, \psi_2 \rangle \overline{\langle \phi_1, \phi_2 \rangle}
    \end{align*}
    from Proposition \ref{prop:wigner_orthog}, and by Cauchy-Schwarz, we must have proportionality in each of the inner products on the right hand side because we have it on the left hand side.
\end{proof}

\subsection{$L^p$ bounds on operator convolutions}\label{sec:Lp_op_conv_bounds}
Using the link between quantization and operator convolutions from quantum harmonic analysis established in Section \ref{sec:operator_convolution_quantization}, we can obtain new results on integrability of operator convolutions. Before stating our result, we briefly survey what bounds have previously been available.

The original article on quantum harmonic analysis by Werner \cite{werner1984} included a version of Young's inequality stated purely for operators, see \cite[Prop.~3.2 (5)]{werner1984}, of the form
\begin{align}\label{eq:werner_young}
    \Vert T \star S \Vert_{L^r(\R^{2d})} \leq \Vert T \Vert_{\mathcal{S}^p} \Vert S \Vert_{\mathcal{S}^q}
\end{align}
where $\frac{1}{p} + \frac{1}{q} = 1 + \frac{1}{r}$. It should be stressed that this was only shown for the Weyl-Heisenberg group.

In the generalization of quantum harmonic analysis to the affine group and locally compact groups \cite{Berge2022, Halvdansson2023}, the $p = q = r = 1$ base case of \eqref{eq:werner_young} needed to take into consideration the notion of admissibility, yielding $\Vert T \star S \Vert_{L^1_r(G)} \leq \Vert T \Vert_{\mathcal{S}^1} \Vert \D^{-1} S \D^{-1} \Vert_{\mathcal{S}^1}$. Interpolating this leads to the inequality
\begin{align*}
    \Vert T \star S \Vert_{L^p_r(G)} \leq \Vert T \Vert_{\mathcal{S}^p} \Vert S \Vert_{\mathcal{S}^1}^{\frac{1}{q}} \big\Vert \D^{-1} S \D^{-1} \big\Vert_{\mathcal{S}^1}^{1/p}
\end{align*}
for $\frac{1}{p} + \frac{1}{q} = 1$, see \cite[Prop.~4.13 (ii)]{Halvdansson2023} or \cite[Prop.~4.18 (2)]{Berge2022} for the proof.

Interpolating the mapping $K_S : T \mapsto T \star S$ between the above and the standard Schatten norm inequality $\Vert T \star S \Vert_{L^\infty} \leq \Vert T \Vert_{\mathcal{S}^q} \Vert S \Vert_{\mathcal{S}^p}$ yields
\begin{align*}
    \Vert T \star S \Vert_{L^{q_\theta}_r(G)}\leq \Big( \Vert S \Vert_{\mathcal{S}^1}^{1/q} \big\Vert \D^{-1} S \D^{-1} \big\Vert_{\mathcal{S}^1}^{1/p} \Big)^{1-\theta} \Vert S \Vert_{\mathcal{S}^p}^\theta \Vert T \Vert_{\mathcal{S}^{p_\theta}}
\end{align*}
where $\frac{1}{p} + \frac{1}{q} = 1$, $\theta \in (0,1)$, $\frac{1}{p_\theta} = \frac{\theta
}{q} + \frac{1-\theta}{p}$ and $\frac{1}{q_\theta} = \frac{1-\theta}{p}$. However, we do not get the standard $1+\frac{1}{r} = \frac{1}{p} + \frac{1}{q}$ structure from Young's inequality due to how we must treat the admissibility constant.

These approaches have all been operator-based and we now turn our attention to the construction of a quantization-based bound which essentially follows from Young's inequality for locally compact groups.
\begin{proposition}
    Let $T, S \in \mathcal{S}^2$ and $\frac{1}{p} + \frac{1}{q} = 1 + \frac{1}{r},\,\frac{1}{q} + \frac{1}{q'} = 1$ with $1 \leq p, q, q', r \leq \infty$, then
    \begin{align*}
        \Vert T \star S \Vert_{L^r_r(G)} \leq \big\Vert a_T \Delta^{1/q'}\big\Vert_{L_r^p(G)} \big\Vert a_S \big\Vert_{L_l^q(G)}.
    \end{align*}
\end{proposition}
\begin{proof}
    The key tool we will use is Young's inequality for locally compact groups \cite[Lem.~2.1]{Klein1978} which can be stated as
    \begin{align}\label{eq:lc_young}
        \Vert f  *_G g \Vert_{L_r^r(G)} \leq \big\Vert f \Delta^{1/q'} \big\Vert_{L_r^p(G)} \Vert g \Vert_{L_r^q(G)}
    \end{align}
    with $p,q,q',r$ as in the proposition. Proposition \ref{prop:op_op_quantization} now immediately yields that
    \begin{align*}
        \Vert T \star S \Vert_{L_r^r(G)} = \Vert a_T * \check{a_S} \Vert_{L_r^r(G)}
    \end{align*}
    and so our quantity is in the form of \eqref{eq:lc_young}. We set $f = a_T,\, g = \check{a_S}$ and apply the inequality to obtain
    \begin{align*}
        \Vert T \star S \Vert_{L_r^r(G)} &\leq \big\Vert a_T \Delta^{1/q'} \big\Vert_{L_r^p(G)} \Vert \check{a_S}\Vert_{L_r^q(G)} = \big\Vert a_T \Delta^{1/q'} \big\Vert_{L_r^p(G)} \Vert a_S \Vert_{L_l^q(G)}
    \end{align*}
    which is what we wished to show.
\end{proof}

\bibliographystyle{abbrv}
\bibliography{ref}
\Addresses

\end{document}